 \documentclass[leqno]{amsart}
\usepackage[T1]{fontenc}
\usepackage{amsfonts,amssymb,amsmath,amsgen,amsthm}
\usepackage{hyperref,color}
\usepackage{pdfsync}
\usepackage{mathtools}
\usepackage{orcidlink}
\makeatletter
\theoremstyle{plain}
\newtheorem{thm}{\protect\theoremname}
\theoremstyle{definition}
\newtheorem{defn}[thm]{\protect\definitionname}
\theoremstyle{remark}
\newtheorem{rem}[thm]{\protect\remarkname}
\theoremstyle{plain}
\newtheorem{lem}[thm]{\protect\lemmaname}
\theoremstyle{plain}
\newtheorem{prop}[thm]{\protect\propositionname}

\makeatother

\usepackage{babel}
\providecommand{\definitionname}{Definition}
\providecommand{\lemmaname}{Lemma}
\providecommand{\propositionname}{Proposition}
\providecommand{\remarkname}{Remark}
\providecommand{\theoremname}{Theorem}

\def\R{{\mathbf R}}
\def\T{{\mathbb T}}
\def\N{{\mathbf N}}
\def\Z{{\mathbf Z}}
\def\d{{\partial}}
\def\eps{\varepsilon}
\DeclareMathOperator{\RE}{Re}
\DeclareMathOperator{\IM}{Im}
\DeclareMathOperator{\diver}{div}

\numberwithin{equation}{section}

\date\today

\title{Large-Norm Solutions and the Relaxation-Time Limit for Quantum Hydrodynamics on the Two-Dimensional Torus}

\author[H. Zheng]{Hao Zheng\,\orcidlink{0000-0002-7098-7731}}
\address{Chinese Academy of Science, Zhongguancun Est. Rd., Haidian District, Beijing, China, 100190}
\email{zhenghao@amss.ac.cn}


\begin{document}
\begin{abstract}
This paper extends to the two-dimensional torus our previous analysis \cite{AMZ3} of weak solutions with large norms for the collisional quantum hydrodynamic (QHD) system in semiconductor modeling.

We first establish the global well-posedness of weak solutions with strictly positive density within the functional framework of generalized chemical potential (GCP) solutions introduced in \cite{AMZ1}. Two key ingredients of the analysis are a logarithmic Sobolev-type inequality controlling oscillations of the density and a functional combining a higher-order energy with the physical entropy. This combined functional yields a coercive dissipation mechanism that allows us to establish stability and exponential convergence for solutions with large initial data. As a byproduct of our approach, we also prove the global existence of $H^2$ solutions for a nonlinear Schr\"odinger--Langevin equation.

Finally, for GCP solutions with strictly positive density, we justify the relaxation-time limit and provide an explicit convergence rate. Our analysis relies on compactness techniques that do not require the existence or smoothness of solutions to the limiting system. Moreover, our results impose no well-preparedness assumptions on the initial data, thereby accommodating the possible formation of an initial layer.
\end{abstract}
\maketitle

\section{Introduction and main results}

In this paper, we focus on weak solutions with large norms for the quantum hydrodynamic (QHD) system arising in semiconductor device modeling on the two-dimensional torus, and on their relaxation-time limit. The quantum hydrodynamic system is given by the following hydrodynamic equations:
\begin{equation}\label{eq:QHD}
\begin{cases}
\d_{t}\rho+\diver J=0\\
\d_{t}J+\diver\left(\frac{J\otimes J}{\rho}\right)+\nabla p(\rho)+\rho\nabla V=\frac{1}{2}\rho\nabla\left(\frac{\triangle\sqrt\rho}{\sqrt\rho}\right)-\frac{1}{\tau}J\\
-\triangle V=\rho-\mathcal{C}(x),\quad (\rho,J)(0,x)=(\rho_0,J_0)(x),
\end{cases}
\end{equation}
for $(t,x)\in [0,\infty)\times \T^2$. Here $\T^2$ denotes the two-dimensional torus, which we assume, without loss of generality, to be $\T^2=\R^2/\Z^2$. The quantum term can be equivalently written as
\begin{equation}\label{eq:bohm}
\frac12\rho\nabla\left(\frac{\triangle\sqrt{\rho}}{\sqrt{\rho}}\right)=\frac14\diver(\rho\nabla^2\log\rho)=\frac14\nabla\triangle\rho-\diver(\nabla\sqrt{\rho}\otimes\nabla\sqrt{\rho}),
\end{equation}
where the second expression corresponds to a quantum stress tensor, and the last identity exhibits the dispersive effect of the quantum term together with the nonlinear effect of Fisher information \cite{HC}. 

The quantum hydrodynamic model \eqref{eq:QHD} finds application in semiconductor device modeling \cite{J}. In this context, the function $V$ represents a self-consistent electric potential governed by the Poisson equation, while $\mathcal{C}(x)$ denotes the density of stationary positive background charges. Regarding the origin of the collisional term in semiconductor physics, its derivation can proceed via a phenomenological approach, accounting for thermal interactions among charge carriers (see \cite{AI,AT,G} or Section 7 in \cite{AM_b}). Alternatively, it can be obtained from kinetic models as an approximation for intra-band collision processes \cite{BW}. For mathematical studies of quantum hydrodynamics in semiconductor modeling, we refer to \cite{Guo} and references therein. 

Such quantum fluid models are also applied to diverse physical systems, particularly in regimes where the thermal de Broglie wavelength is of the order of the interatomic distance. Relevant examples include superfluidity in helium II \cite{K}, Bose-Einstein condensates \cite{PS}, and quantum plasmas \cite{H}. Furthermore, the collisional model \eqref{eq:QHD} with $\tau>0$ serves as a toy model for examining interactions between quantum and classical fluids. This is conceptually related to the two-fluid model introduced by Landau and Tisza; see the discussions in \cite[Section 17]{K} and \cite[Chapter XVI]{LL}. 

The main purpose of this paper is to establish the global well-posedness of weak solutions of \eqref{eq:QHD} with large norms and strictly positive density $\rho$. Moreover, we will study the relaxation-time limit of the system \eqref{eq:QHD}. By introducing the scaling \cite{DM,MMS,MN}
\begin{equation}\label{eq:rs_intro}
t'=\tau t,\quad (\rho_\tau,J_\tau)(t',x)=\left(\rho,\frac{1}{\tau}J\right)\left(\frac{t'}{\tau},x\right),
\end{equation}
the system \eqref{eq:QHD} can be reformulated as
\begin{equation}\label{eq:QHD_rs_intro}
\left\{\begin{aligned}
&\d_{t'}\rho_\tau+\diver J_\tau=0\\
&\tau^2\d_{t'} J_\tau+\tau^2\diver\left(\frac{J_\tau\otimes J_\tau}{\rho_\tau}\right)+\nabla p(\rho_\tau)+\rho_\tau\nabla V_\tau=\frac{1}{2}\rho_\tau\nabla\left(\frac{\triangle\sqrt\rho_\tau}{\sqrt\rho_\tau}\right)-J_\tau\\
&-\triangle V_\tau=\rho_\tau-\mathcal{C}(x),\quad (\rho_\tau,J_\tau)(0,x)=(\rho_0,J_{\tau,0})(x),
\end{aligned}\right.
\end{equation}
As $\tau\to 0$, the system \eqref{eq:QHD_rs_intro} formally converges to following the quantum drift--diffusion (QDD) system
\begin{equation}\label{eq:qdde}
\begin{cases}
\d_{t'}\bar\rho+\diver\left[\frac{1}{2}\bar\rho\nabla\left(\frac{\triangle\sqrt{\bar\rho}}{\sqrt{\bar\rho}}\right)-\nabla p(\bar\rho)-\bar\rho\nabla\bar{V}\right]=0\\
-\triangle\bar V=\bar\rho-\mathcal{C}(x),\quad \bar\rho(0,x)=\rho_0(x),
\end{cases}
\end{equation}
see for instance \cite{J, R}. We also rigorously justify this limit in a suitable functional framework that will be described below.

A natural framework for studying the system \eqref{eq:QHD} is the class of solutions with finite mass and finite energy, characterised by the conserved total mass
\begin{equation*}
M(t)=\int_{\T^2} \rho(t)dx\equiv M_0,
\end{equation*}
and the total energy functional
\begin{equation}\label{eq:en}
E(t)=\int_{\T^2} \frac12|\nabla\sqrt\rho|^2+\frac{|J|^2}{2\rho}+f(\rho)+\frac12|\nabla V|^2dx.
\end{equation}
In the collisionless case ($\tau=\infty$), $E(t)$ is the conserved energy associated with the Hamiltonian QHD system, while in the case $0<\tau<\infty$, it satisfies the corresponding energy balance law: 
\begin{equation}\label{eq:en_disp_0}
E(t)+\frac{1}{\tau}\int_0^t\int_{\T^2} \frac{|J|^2}{\rho}dxds=E(0).
\end{equation}
The internal energy $f(\rho)$ is related to the pressure $p(\rho)$ via the relation
\begin{equation}\label{eq:pres}
p(\rho)=f'(\rho)\rho-f(\rho),
\end{equation}
or equivalently
\[
f(\rho)=\rho\int^\rho \frac{p(s)}{s^2}ds.
\]
In this paper, we consider the family of pressure laws generated by
\begin{equation}\label{eq:gamma-2n}
f(\rho)=\frac{1}{2n}(\rho-M_0)^{2n},\quad n\in \N_+.
\end{equation}
For the Poisson equation in \eqref{eq:QHD} governing the electric potential $V$, we also assume $\mathcal{C}(x)=M_0$, which guarantees its solvability. 

The existence of finite-energy weak solutions has been established by combining the results of \cite{AM1, AM2} and \cite{AMZ1}. However, the uniform bounds provided solely by the energy functional are not sufficient to establish a compactness framework for \eqref{eq:QHD_rs_intro}. Indeed, while the inertial term can be controlled by the scaled energy dissipation, the weak limit of the quadratic term $(\nabla\sqrt{\rho_\tau}\otimes \nabla\sqrt{\rho_\tau})$ (see \eqref{eq:bohm} above) may be affected by strong oscillations and concentration phenomena, which cannot be controlled by the energy bounds alone.

Therefore, in order to overcome this difficulty, we apply the compactness framework developed in the authors' previous work \cite{AMZ1} and introduce the notion of weak solutions with bounded generalized chemical potential (see Definition \ref{def:GCPsln} below), hereafter referred to as \emph{GCP solutions}. The chemical potential is formally defined as
\[
\mu=-\frac{\triangle\sqrt\rho}{2\sqrt\rho}+\frac12\frac{|J|^2}{\rho^2}+f'(\rho)+V,
\]
which corresponds formally to the first variation of the total energy functional with respect to the mass density,
\[
\mu=\frac{\delta E}{\delta \rho}.
\]
Indeed, by introducing $v=\frac{J}{\rho}$ and noticing that $J=\frac{\delta E}{\delta v}$, equations \eqref{eq:QHD} can be formally written as the following system (Hamiltonian system with a damping term) associated with the energy functional
\begin{equation}\label{eq:Ham_form}
\left\{\begin{aligned}
&\d_t\rho=-\diver \frac{\delta E}{\delta v}\\
&\d_t v=-\nabla \frac{\delta E}{\delta \rho}-\frac{1}{\tau}v
=-\nabla\mu-\frac1\tau v.
\end{aligned}\right.
\end{equation}
The preceding quantities allow us to define the GCP functional
\begin{equation}\label{eq:higher}
I(t)=\int_{\T^2} \frac{\rho}{2}(\mu^2+\sigma^2)dx,
\end{equation}
where 
\[
\sigma=\d_t\log\sqrt\rho=-\frac{\diver J}{2\rho}.
\]

The functional $I(t)$ characterizes higher-order regularity for solutions of \eqref{eq:QHD}. In \cite{AMZ1,AMZ2}, it is shown that $I(t)$ remains uniformly bounded on compact time intervals for the collisionless QHD system. Moreover, a compactness class for weak solutions to QHD is identified in the framework of $I(t)$. 
Following the terminology of \cite{AMZ1,AMZ2}, we refer to weak solutions with finite energy (see \eqref{eq:en}) and finite GCP functional (defined in \eqref{eq:higher}) as GCP solutions, see also Definition \ref{def:GCPsln}.
Unlike \cite{AMZ1,AMZ2} where vacuum regions are allowed, 
our analysis requires considering solutions whose density remains uniformly bounded away from vacuum.
Indeed, by using \eqref{eq:QHD}, the time derivative of $I(t)$ is formally given by
\begin{equation}\label{eq:dtI_intro}
\frac{d}{dt} I(t)+\frac{1}{\tau}\int_{\T^2}\rho \sigma^2 dx=\int_{\T^2} \mu\d_tp(\rho)dx+\int_{\T^2}\rho\mu \d_tVdx-\frac{1}{\tau}\int_{\T^2}\rho |v|^2\mu dx.
\end{equation}
In particular, the last term on the right-hand side of \eqref{eq:dtI_intro} appears due to the collision term. Strong density fluctuations near vacuum may cause the chemical potential to become highly singular and may lead to a loss of integrability. Therefore in this paper we restrict to strictly positive densities. A rigorous derivation of $\frac{d}{dt}I(t)$ will be given in Proposition \ref{prop:3.2}.

The GCP solution framework was successfully applied to \eqref{eq:QHD} in the one-dimensional case in \cite{AMZ3}, where the strict positivity of the mass density $\rho$ is ensured by combining the Poincar\'e inequality with the energy balance law \eqref{eq:en_disp_0}. In higher dimensions, however, the energy estimate alone proves insufficient to regulate the lower bound of $\rho$. To overcome this difficulty, we introduce a logarithmic Sobolev-type inequality (Lemma \ref{lem:logsob}), which controls the oscillation of $\rho$ in terms of $E(t)$ and $I(t)$. In order to maintain strict positivity of $\rho$, we impose the condition \eqref{eq:cond_2} on the initial energy $E(0)$ and the initial GCP functional $I(0)$. We emphasize, nevertheless, that this condition does not exclude initial data with large norms; see Remark \ref{rem:r1}.

\subsection{Main result: global well-posedness of GCP solutions and stability}

The first result of this paper concerns the global well-posedness of GCP weak solutions with strictly positive density. Let us first put this result in perspective with the existing smooth theory for QHD systems. Global smooth solutions and related stability results have been obtained for small perturbations of an equilibrium; see, for instance, \cite{RaHong,JLi,LiMarcati,LZZ} and the references therein. In these works, the smallness of the initial data plays a crucial role in controlling the nonlinear interactions. The restriction imposed in Theorem \ref{thm:glob2}, however, is of a different nature. Although condition \eqref{eq:cond_2} involves the initial energy $E_0$ and the initial GCP functional $I_0$, its right-hand side tends to infinity as the total mass $M_0$ tends to infinity, for every fixed $\delta>0$ (see Remark \ref{rem:r1}). Consequently, \eqref{eq:cond_2} does not impose a uniform smallness assumption on the initial norms. Rather, it allows initial data with arbitrarily large GCP norm, provided that the mass is sufficiently large. The key point is that our argument does not rely on a generic smallness condition to absorb the nonlinear terms. Instead, we identify a dissipative structure arising from the combined use of the higher-order energy and the physical entropy, which makes it possible to control the nonlinear interactions. The only remaining mechanism that may lead to a singularity is the formation of vacuum; this is precisely why the positivity condition encoded in \eqref{eq:cond_2} is imposed.

Recall that $M_0=\int_{\T^2}\rho_0dx<\infty$ is the conserved total mass, and denote by $E_0=E(0)$ and $I_0=I(0)$ the total energy and the higher-order functional at the initial time, defined by \eqref{eq:en} and \eqref{eq:higher}, respectively.

\begin{thm}[Global well-posedness of GCP weak solutions]\label{thm:glob2}
Let us consider a finite-energy initial datum $(\rho_0, J_0)$ satisfying the following conditions.
\begin{itemize}
\item The initial density satisfies $\inf_x\rho_0\geq \delta$ for some $0<\delta<M_0$. Let us denote by $v_0=J_0/\rho_0$ the initial velocity. We assume that  $\operatorname{curl}v_0=0$ and
\begin{equation}\label{eq:cond_1}
\int_\T v_0(y,x_2) dy=\int_\T v_0(x_1,y) dy=0
\end{equation}
for almost every $(x_1,x_2)\in\T^2$.

\item For a universal constant $\eps_0>0$, independent of the initial data, assume that
\begin{equation}\label{eq:cond_2}
e^{E_0}\cdot (1+I_0)\leq \eps_0\frac{e^{M_0-\delta}}{(M_0-\delta)^{2n}}.
\end{equation}
\end{itemize}
Then there exists $\tau^*>0$, depending on $(\delta,M_0,E_0,I_0)$, such that for $0<\tau\leq \tau^*$, the Cauchy problem for the QHD system \eqref{eq:QHD} has a unique global-in-time GCP solution $(\rho,J)$ with the following properties for every $0\leq t<\infty$:
\begin{itemize}
\item the density remains strictly positive with lower bound $\inf_{t,x}\rho\geq \delta$;
\item the total energy $E(t)$ satisfies the energy balance law, namely for every $t>0$,  
\begin{equation}\label{eq:en_disp_intro}
E(t)+\frac{1}{\tau}\int_0^t\int_{\T^2} \rho |v|^2 dxds=E_0;
\end{equation}
\item the higher-order functional $I(t)$ satisfies the bounds
\begin{equation}\label{eq:bd_I}
I(t)+\frac{1}{\tau}\int_0^t\int_{\T^2} [(\d_t\sqrt\rho)^2+\rho |v|^4 ]dxds\leq C(\delta,M_0,E_0,I_0).
\end{equation}
\end{itemize}
\end{thm}

\begin{rem}\label{rem:r1}
The details of the condition \eqref{eq:cond_2} and the constant $\eps_0>0$ are provided in Section \ref{sect:exist2d} through the proof of Theorem \ref{thm:glob2}; see, for example, \eqref{eq:dep_2d}. In particular, for fixed $\delta>0$, the right-hand side of \eqref{eq:cond_2} tends to infinity as $M_0\to\infty$.
\end{rem}

Our proof of Theorem \ref{thm:glob2} is motivated by the formal equivalence between system \eqref{eq:QHD} and the following Schr\"odinger--Langevin-type equation:
\begin{equation}\label{eq:NLS_intro}
i\d_t\psi+\frac12\triangle\psi=f'(|\psi|^2)\psi+\frac{1}{\tau}S\psi+V\psi,
\end{equation}
see also \cite{JMR}. Here $S$ denotes the phase associated with the Madelung representation $\psi=|\psi|e^{iS}$ and can be formally written as
\[
S=\frac{1}{i}\log\left(\frac{\psi}{|\psi|}\right).
\]
The hydrodynamic variables are recovered from $\psi$ by
\[
\rho=|\psi|^2,\qquad J=\IM(\bar\psi\nabla\psi).
\]
Then the hydrodynamic system \eqref{eq:QHD} can be formally recovered from \eqref{eq:NLS_intro} by computing the associated balance laws; see, for instance, \cite{AMZ1}. 
Moreover, for strictly positive densities, the GCP quantities are related to $\psi$ through
\[
\mu=-\frac12\RE\left(\frac{\triangle\psi}{\psi}\right)+f'(|\psi|^2)+V,
\qquad
\sigma=-\frac12\IM\left(\frac{\triangle\psi}{\psi}\right).
\]
Hence, one expects the energy functional \eqref{eq:en} and the GCP functional \eqref{eq:higher} to be comparable to the first- and second-order Sobolev norms of $\psi$, respectively:
\[
E(t)\sim \|\nabla\psi\|_{L^2_x}^2,
\qquad
I(t)\sim \|\nabla^2\psi\|_{L^2_x}^2.
\]
Later, we shall rigorously justify these equivalences and establish the correspondence between the GCP framework for \eqref{eq:QHD} and the classical Sobolev theory associated with \eqref{eq:NLS_intro}.

Equation \eqref{eq:NLS_intro} was also independently introduced, in the case $f=0$, by Kostin \cite{Kos} (see also \cite{Na} and \cite{Y}) to derive a Schr\"odinger equation that takes into account the interaction with Brownian particles. The multivalued nature of the complex logarithm prevents us from studying equation \eqref{eq:NLS_intro} with standard methods.

To prove Theorem \ref{thm:glob2}, we first show the well-posedness of the Cauchy problem for the Schr\"odinger--Langevin equation \eqref{eq:NLS_intro}. Then we rigorously establish the equivalence between the QHD system \eqref{eq:QHD} and equation \eqref{eq:NLS_intro} for solutions with strictly positive density, using the methods of wave function lifting \cite{AMZ1,AMZ2} (see also \cite{HaoMinE}) and polar factorization \cite{AM1} developed in the authors' previous works. As an auxiliary result, we also state the well-posedness of equation \eqref{eq:NLS_intro} as the following theorem.

\begin{thm}\label{thm:NLS}
Assume that the initial datum $\psi_0\in H^2_x(\T^2)$ for \eqref{eq:NLS_intro} satisfies 
\[
\inf_x|\psi_0(x)|\geq \delta^\frac12,\quad e^{C_0\|\nabla\psi_0\|_{L^2_x}^2}
\left(1+\|\nabla\psi_0\|_{H^1_x}^2\right)
\le
\eps_0\frac{e^{M_0-\delta}}{(M_0-\delta)^{2 n}}.
\]
Then there exists $\tau^*>0$ such that for $0<\tau\leq \tau^*$, the Cauchy problem associated with \eqref{eq:NLS_intro} with initial data $\psi_0$ has a unique global solution $\psi\in \mathcal{C}([0,\infty);H^2_x(\T^2))$. Moreover, $\inf_{t,x}|\psi|\geq \delta^\frac12$ and the hydrodynamic variables associated with $\psi$, given by
\[
\rho=|\psi|^2,\quad J=\IM(\bar\psi\nabla\psi)=\rho v,
\]
satisfy the bounds \eqref{eq:en_disp_intro} (with $E_0=E(\psi_0)$) and \eqref{eq:bd_I}.
\end{thm}

Additionally, by combining the functionals $E(t)$ and $I(t)$ with the physical entropy
\begin{equation}\label{eq:entr_intro}
H(\rho)=\int_{\T^2}\rho\log\left(\frac{\rho}{M_0}\right)dx,
\end{equation}
we further prove the exponential decay of GCP solutions. This property was also proved for the following fourth-order parabolic equation \cite{DGJ,GST,JM} in the one-dimensional case,
\begin{equation}\label{eq:DLSS}
\d_{t'}\bar\rho+\frac14\diver\cdot\diver(\bar\rho\nabla^2\log\bar\rho)=0,
\end{equation}
which can be viewed as a reduced QDD equation \eqref{eq:qdde} without pressure and electric potential.  Equation \eqref{eq:DLSS} also appears independently in various areas of the mathematical physics literature. It was first derived by Derrida, Lebowitz, Speer, and Spohn in \cite{DLSS1,DLSS2}; therefore, we refer to it as the \emph{DLSS equation}. A similar equation was also obtained in  \cite{Ber, BP1} as a model for thin films. The next result shows that the exponential decay is true also for the hydrodynamic model.

\begin{thm}[Dissipation for small $\tau$]\label{thm:disp}
Let $(\rho,J)$ be a GCP solution to the QHD system \eqref{eq:QHD} provided by Theorem \ref{thm:glob2}. Then there exist constants $C_0>0$ and $c_1,c_2>0$ such that, for $0<\tau\leq\tau^*$, the functional 
\[
F(t)=H(\rho)+E(t)+c_1 I(t)
\]
satisfies
\begin{equation}
F(t)\leq F(0)\exp\left(\frac{C_0}{\tau}\int_0^{t}\int_{\T^2}\rho |v|^2dxds-c_2\tau t\right).
\end{equation}
In particular, $H(\rho)$, $E(t)$ and $I(t)$ decay exponentially for $t$ large.
\end{thm}

\begin{rem}
The precise choices of $C_0>0$ and $c_1,c_2,\tau^*>0$ are given in Proposition \ref{prop:F_2d}. In particular, $c_1,c_2$ and $\tau^*$ are small constants depending on $(\delta,M_0,E_0,I_0)$, while $C_0$ is a universal constant.
\end{rem}

\subsection{Main result: relaxation-time limit}

We now present the result on the relaxation-time limit of the QHD system \eqref{eq:QHD} as $\tau\to 0$. Recall the scaling \eqref{eq:rs_intro}:
\[
t'=\tau t,\quad (\rho_\tau,J_\tau)(t',x)=\left(\rho,\frac{1}{\tau}J\right)\left(\frac{t'}{\tau},x\right),
\]
so that the QHD system \eqref{eq:QHD} can be rewritten as \eqref{eq:QHD_rs_intro}, which formally converges to the QDD equation \eqref{eq:qdde} as $\tau\to 0$. 

The first rigorous justification of the relaxation-time limit for small, smooth perturbations near a constant state was established in \cite{JLM}, and was subsequently extended to the bipolar system (modeling electrons and ions) in \cite{LZZ}. The relaxation-time limit for finite-energy weak solutions was treated in \cite{LT}; there the analysis requires the initial data to be well-prepared, as described below, and imposes sufficient smoothness on solutions to the limiting equation \eqref{eq:qdde}. Similar relaxation-time limit studies have also been carried out in the context of classical hydrodynamic equations; see, for instance, \cite{HL,HP,JP}. See also \cite{ACLS}, where the relaxation-time limit was established for general weak solutions to quantum Navier-Stokes equations.

A critical aspect of the relaxation-time limit is the appearance of an initial layer. The rescaled QHD system \eqref{eq:QHD_rs_intro} requires two pieces of initial data: the density $\rho_0$ and the momentum density $J_{\tau,0}$. By contrast, the QDD equation \eqref{eq:qdde} only prescribes the initial density $\rho_0$, whereas the limiting momentum density must satisfy the constitutive relation
\begin{equation}\label{eq:barJ}
\bar{J}=\frac{1}{2}\bar\rho\nabla\left(\frac{\triangle\sqrt{\bar\rho}}{\sqrt{\bar\rho}}\right)-\nabla p(\bar\rho)-\bar\rho\nabla\bar{V},
\end{equation}
which can be derived formally by letting $\tau\to 0$ in the second equation of \eqref{eq:QHD_rs_intro}.
Initial data for \eqref{eq:QHD_rs_intro} are said to be well-prepared if
\[
J_{\tau,0}=\bar J_0,
\]
where $\bar J_0$ denotes the vector field obtained by substituting
$\bar\rho=\rho_0$ into \eqref{eq:barJ}. 
In the present work, we do not assume the initial data to be well-prepared; consequently, an initial layer emerges naturally. As a result, convergence of the momentum density holds only for $t'>0$.
An additional contribution compared to existing literature is that we show an explicit rate of convergence with respect to the parameter $\tau$.

The proof of Theorem \ref{thm:rlxlimit} relies on deriving suitable a priori bounds for solutions to the hydrodynamic system, in order to pass to the limit (at least in the distributional sense) all terms in equations \eqref{eq:QHD_rs_intro}. These estimates are obtained by analyzing the physical entropy $H(\rho_\tau)$ defined by \eqref{eq:entr_intro}, which provides us with the $L^2_{t'}H^2_x$ bound of $\sqrt\rho_\tau$ uniformly with respect to $\tau$. Moreover, we obtain an explicit convergence rate of this limit with respect to $\tau$ by considering the relative entropy
\begin{align*}
H(\rho_\tau|\bar\rho)(t')=\int_{\T^2} g(\rho_\tau)-g(\bar\rho)-g'(\bar\rho)(\rho_\tau-\bar\rho)dx,
\end{align*}
where $g(s)=s\log s$ and $\bar\rho$ is the limiting solution of the QDD equation \eqref{eq:qdde}.

\begin{thm}[Relaxation-time limit]\label{thm:rlxlimit}
Let $\{(\rho_{\tau},J_{\tau})\}_{\tau>0}$ be a sequence of GCP solutions of the rescaled equation \eqref{eq:QHD_rs_intro}, such that
\begin{itemize}
\item $\inf_{t',x}\rho_{\tau}\geq \delta$ for a uniform constant $\delta>0$,
\item $(\rho_{\tau},J_{\tau})$ satisfy the uniform bounds
\[
E_\tau(t')+\int_0^{t'}\int_{\T^2}\rho_\tau|v_\tau|^2dxds\leq C
\]
and
\[
I_\tau(t')+\int_0^{t'}\int_{\T^2}[(\d_{t'}\sqrt{\rho_\tau})^2+\tau^2\rho_\tau |v_\tau|^4]dxds\leq 
C
\]
for some $0<C<\infty$. 
\end{itemize}
Then the relaxation-time limit holds, namely $\{\rho_{\tau}\}$ converges to a limiting density $\bar\rho$ in the form
\[
\|\rho_\tau-\bar\rho\|_{L^\infty_{t'}L^2_x}+\|\sqrt\rho_\tau (\nabla^2\log\sqrt\rho_\tau-\nabla^2\log\sqrt{\bar\rho})\|_{L^2_{t',x}}\leq C\tau.
\]
Moreover, $\bar\rho$ is a weak solution of \eqref{eq:qdde} in the sense of Definition \ref{def:qdde_ws}. 
\end{thm}

We emphasize that our result applies to general initial data $(\rho_0,J_0)$, without requiring well-prepared initial data. On the other hand, sharper convergence rates can be obtained by combining the relative energy method with the assumption of well-prepared initial data; see for instance \cite{LT}.

Our paper is structured as follows. Definitions and preliminary results are presented in Section 2. Section \ref{sect:dt} is devoted to the computation of the time derivatives of the main functionals using the dynamics of \eqref{eq:QHD} and \eqref{eq:NLS_intro}, while Section \ref{sect:exist2d} is devoted to the proof of the global well-posedness of GCP solutions (Theorem \ref{thm:glob2}) and their exponential dissipation (Theorem \ref{thm:disp}). The equivalence between \eqref{eq:NLS_intro} and \eqref{eq:QHD} is rigorously established through the methods of wave function lifting and polar factorization. Finally, Section \ref{sect:rs} contains the proof of the relaxation-time limit.

\section{Definitions and preliminaries}\label{sect:def}

In this section, we introduce the notation and definitions that will be used throughout this paper.

The Lebesgue and Sobolev norms on $\T^2$ are defined by
\[
\|f\|_{L_{x}^{p}}\coloneqq\left(\int_{\T^2}|f(x)|^{p}dx\right)^{\frac{1}{p}},
\]
\[
\|f\|_{W_{x}^{k,p}}\coloneqq\sum_{|\alpha|\leq k}\|\partial_{x}^{\alpha}f\|_{L_{x}^{p}},
\]
and $H_{x}^{k}\coloneqq H^{k}(\T^2)$ denotes the Sobolev space $W^{k,2}(\T^2)$. 
Given a time interval $I\subset[0,\infty)$, the mixed space-time Lebesgue norm of a function $f:I\to L^r(\T^2)$ is defined by
\[
\|f\|_{L_{t}^{q}L_{x}^{r}}\coloneqq\left(\int_{I}||f(t)||_{L_{x}^{r}}^{q}dt\right)^{\frac{1}{q}}=\left(\int_{I}\left(\int_{\T^2}|f(t,x)|^{r}dx\right)^{\frac{q}{r}}dt\right)^{\frac{1}{q}}.
\]
Similarly, the mixed Sobolev norm $L_{t}^{q}W_{x}^{k,r}$ is defined. We use $C_0>0$ to denote a generic constant that may change from line to line, and $C(A)$ indicates dependence on the quantity $A$.

Assuming $\rho>0$ and defining the velocity by
$v=\frac{J}{\rho}$,
system \eqref{eq:QHD} can be rewritten as
\begin{equation}\label{eq:QHD_md}
\left\{\begin{aligned}
&\d_t\rho+\diver(\rho v)=0\\
&\d_t v+(v\cdot\nabla)v+\nabla f'(\rho)+\nabla V=\frac12\nabla\left(\frac{\triangle\sqrt\rho}{\sqrt\rho}\right)-\frac{1}{\tau} v.
\end{aligned}\right.
\end{equation}
Let us also recall that the total mass and total energy are respectively given by
\begin{equation}\label{eq:mass}
M(t)=\int_{\T^2}\rho(t, x)\,dx,
\end{equation}
and
\begin{equation}\label{eq:en_md}
E(t)=\int_{\T^2}\frac12\rho |v|^2+\frac12|\nabla\sqrt{\rho}|^2+f(\rho)+\frac12|\nabla V|^2dx.
\end{equation}
We also recall the generalized chemical potential associated with GCP solutions:
\begin{equation}\label{eq:chem}
\mu=-\frac{\triangle\sqrt\rho}{2\sqrt\rho}+\frac12|v|^2+f'(\rho)+V,\quad \sigma=\d_t\log\sqrt{\rho}=-\frac{\diver J}{2\rho}.
\end{equation}
The $L^2$ norms of $\mu$ and $\sigma$ with respect to the measure $\rho\,dx$ defines the GCP functional
\begin{equation}\label{eq:GCP}
I(t)=\int_{\T^2}\frac12\rho(\mu^2+\sigma^2)dx.
\end{equation}
Let us remark that this definition may apply to a more general class of weak solutions, which includes states with vacuum regions; see \cite{AMZ1,AMZ2} for more details.
The definition of weak solutions with bounded generalized chemical potential is given below.

\begin{defn}[GCP solutions]\label{def:GCPsln}
We say that $(\rho, v)$ is a GCP solution to the Cauchy problem for the QHD system \eqref{eq:QHD_md} on $[0,T]\times\T^2$ if
\[
(\sqrt{\rho},\sqrt{\rho}v)
\in L^1([0,T);H^1(\T^2)\times L^2(\T^2))
\]
satisfies system \eqref{eq:QHD_md} in the sense of distributions, and the following estimates hold:
\begin{equation*}
\begin{aligned}
\|\sqrt{\rho}\|_{L^\infty(0, T; H^1(\T^2))}+\|\sqrt\rho v\|_{L^\infty(0, T; L^2(\T^2))}\leq& C_0,\\
\|\d_t\sqrt{\rho}\|_{L^\infty(0, T; L^2(\T^2))}+\|\sqrt\rho\mu\|_{L^\infty(0, T; L^2(\T^2))}\leq& C_0.
\end{aligned}
\end{equation*}
Moreover, if $T=\infty$, we call $(\rho, v)$ a global solution.
\end{defn}

Next, we introduce the physical entropy $H(\rho)$ of a density function $\rho\geq0$ on $[0,T)\times\T^2$, defined by
\begin{equation}\label{eq:entr}
H(\rho)=\int_{\T^2}\rho\log\left(\frac{\rho}{M_0}\right)dx,
\end{equation}
where $M_0= M(t)$ is the conserved total mass of $\rho$ given by \eqref{eq:mass}. Since we assume $|\T^2|=1$, $M_0$ can also be viewed as the average of $\rho$. Therefore, by convexity of the function $g(s)=s\log s$, we have 
\[
H(\rho)=\int_{\T^2} \rho\log\rho dx-M_0\log M_0\geq 0.
\]
For later use in the analysis of the relaxation-time limit, we now give the definition and dissipative property of weak solutions to the quantum drift-diffusion equation \eqref{eq:qdde}.

\begin{defn}\label{def:qdde_ws}
We say that $\bar\rho$ is a weak solution to the quantum drift-diffusion equation \eqref{eq:qdde} with initial data $\bar\rho(0)=\rho_0\in L^1(\T^2)$ on $[0,T)\times\T^2$, if $\sqrt{\bar\rho}\in L^2_{\mathrm{loc}}([0,T);H^1(\T^2))$
and satisfies equation \eqref{eq:qdde} in the sense of distributions. 

Moreover, $\bar\rho$ is called a dissipative solution if 
\begin{equation}\label{eq:disp_entrp_1}
H(\bar\rho)(t)+\int_0^t\int_{\T^2}|\nabla^2\sqrt{\bar\rho}|^2+|\nabla(\bar\rho^\frac14)|^4dxds\leq H(\rho_0)
\end{equation}
for any $0\leq t<T$.
\end{defn}

As discussed in the Introduction, there is a formal analogy between system \eqref{eq:QHD_md} and the Schr\"odinger--Langevin type equation 
\begin{equation}\label{eq:NLS}
i\d_t\psi+\frac12\triangle\psi=f'(|\psi|^2)\psi+\frac{1}{\tau}S\psi+V\psi,
\end{equation}
through the so-called Madelung transformation \cite{Mad}. Indeed, this can be seen by writing $\psi$ in polar form 
\begin{equation}\label{eq:mad}
\psi=\sqrt{\rho}e^{iS}.
\end{equation}
Substituting \eqref{eq:mad} into \eqref{eq:NLS} and separating the real and imaginary parts of the resulting equation, one obtains the system for $\rho$ and $S$: 
\begin{equation}\label{eq:qHJ}
\left\{\begin{aligned}
&\d_t\rho+\diver(\rho\nabla S)=0\\
&\d_tS+\frac12|\nabla S|^2+f'(\rho)+V+\frac1\tau S=\frac12\frac{\triangle\sqrt{\rho}}{\sqrt{\rho}}.
\end{aligned}\right.
\end{equation}
If we further differentiate the second equation in \eqref{eq:qHJ} with respect to $x$, we obtain an Euler-type equation for the velocity field $v=\nabla S$, namely
\begin{equation}\label{eq:qvel}
\operatorname{curl}v=0,\;\d_tv+\nabla\left(\frac12|v|^2+f'(\rho)+V-\frac12\frac{\triangle\sqrt{\rho}}{\sqrt{\rho}}\right)=-\frac1\tau v.
\end{equation}
By recalling the definition of the chemical potential in \eqref{eq:chem}, we see that the evolution equation in \eqref{eq:qvel} coincides with the second equation in \eqref{eq:Ham_form}.
The momentum equation in the QHD system \eqref{eq:QHD} is then derived by multiplying \eqref{eq:qvel} by $\rho$ and by using the continuity equation. 

Following the discussion above, we say that the hydrodynamic variables associated with $\psi\in H_x^s$, $s\ge1$, are given by the pair $(\rho,\rho v)$ if
\begin{equation}\label{eq:assoc}
(\rho,\rho v)=(|\psi|^2,\IM(\bar\psi\nabla\psi)).
\end{equation}
If one further has $\inf|\psi|>0$, then the velocity can be directly written as
\begin{equation}\label{eq:veloc}
v=\IM\left(\frac{\nabla\psi}{\psi}\right).
\end{equation}

Conversely, one may recover a complex wave function $\psi$ from given hydrodynamic variables $(\rho,v)$ provided that $\rho\geq\delta>0$. Motivated by \eqref{eq:qHJ}, we define
\begin{equation}\label{eq:gradS}
\begin{cases}
\nabla S =v,\quad\d_t S=-\mu-\frac{1}{\tau}S \\
S(0,0)=S_*.
\end{cases}
\end{equation}
The equation \eqref{eq:qvel} for the velocity field may be interpreted as the "irrotationality" condition for the gradient equation \eqref{eq:gradS}, namely
\begin{equation}\label{eq:qvel_irrot}
\operatorname{curl}v=0,\quad\d_tv+\nabla(\mu+\frac1\tau S)=0.
\end{equation}
Thus we can formally write the phase function $S$ in terms of the hydrodynamic variables as
\begin{equation}\label{eq:defS_pt}
S(t,x)=e^{-\frac{t}{\tau}} \left(S_*+\int_{l(0,x_*)}v_0(y)\cdot d\vec{l}(y)\right)-\int_0^t e^{\frac{s-t}{\tau}}\mu(s,x_*)ds+\int_{l(x_*,x)}v(t,y)\cdot d\vec{l}(y),
\end{equation}
where $x_*\in \T^2$ is arbitrary and $l(x,y)$ is a piecewise smooth curve connecting $x$ and $y$. For example, we can choose $l(x,y)$ to be the collection of piecewise straight lines parallel to each axis, namely for $x=(x_1,x_2)$, $y=(y_1,y_2)$, we have
\begin{equation}\label{eq:l}
l(x,y)=\{(s,x_2);s\textrm{ from }x_1\text{ to }y_1\}\cup\{(y_1,s);s\text{ from }x_2\textrm{ to }y_2\}.
\end{equation}
Due to the "irrotationality" condition \eqref{eq:qvel_irrot}, the choice of point $x_*$ has no influence on the definition. However, for weak solutions $(\rho,v)$, the chemical potential $\mu$ may only belong to a Lebesgue space, or even be a distribution. Since the right-hand side of \eqref{eq:defS_pt} should be independent of the choice of $x_*$, we can take the average of \eqref{eq:defS_pt} with respect to $x_*$, thereby allowing us to extend the hydrodynamic definition of the phase function $S$ to more general cases.

\begin{defn}[Phase function]\label{def:phase}
Let $(\rho,v)$ be a solution of \eqref{eq:QHD_md} on $[0,T)\times\T^2$ such that the velocity field $v=(v_1,v_2)^t$ and the chemical potential $\mu$ given by \eqref{eq:chem} satisfy:
\begin{equation}\label{eq:condphase}
v\in L^\infty_tL^1_x,\quad\mu\in \mathcal{D}'([0,T)\times\T^2),
\end{equation}
and
\begin{equation}\label{eq:avrg0}
\operatorname{curl} v=0,\quad\int_\T v_2(t,x_1,y) dy=\int_\T v_1(t,y,x_2) dy=0,
\end{equation}
for almost every $t\in[0,T)$ and $(x_1,x_2)\in\T^2$. 
Then we define the phase function $S$ to be the solution of the gradient equation \eqref{eq:gradS}, which can be explicitly written as
\begin{equation}\label{eq:defS}
\begin{aligned}
S(t,x)=&e^{-\frac{t}{\tau}} \left(S_*+\int_{\T^2}\int_{l(0,x_*)}v_0(y)\cdot d\vec{l}(y)dx_*\right)\\
&-\int_0^t e^{\frac{s-t}{\tau}} \int_{\T^2} \mu(s,x_*)dx_*ds+\int_{\T^2}\int_{l(x_*,x)}v(t,y)\cdot d\vec{l}(y)dx_*.
\end{aligned}
\end{equation}
\end{defn}

\begin{rem}\label{rem:avrg0}
The zero-average condition in \eqref{eq:avrg0} is needed to ensure that $S$ is periodic on $\T^2$; however, Definition \ref{def:phase} remains valid without this requirement. Moreover, condition \eqref{eq:avrg0} is preserved by the dynamics of \eqref{eq:QHD_md}. Indeed, if the initial velocity $v_0$ satisfies condition \eqref{eq:avrg0}, then by taking, respectively, the curl and the directional averages of the second equation in \eqref{eq:QHD_md}, one can show that \eqref{eq:avrg0} also holds for the solution $v(t)$.
\end{rem}

\begin{rem}\label{rem:equivS}
Let us remark that the assumption $|\psi|\geq\sqrt{\delta}$ implies that a continuous branch of the logarithm $S=\frac{1}{i}\log\left(\frac{\psi}{|\psi|}\right)$ is well defined. Consequently, by the uniqueness of equation \eqref{eq:gradS}, it coincides with \eqref{eq:defS}, for a suitable choice of $S_*$.
\end{rem}

Using the definition \eqref{eq:chem} of $\mu$, after taking the average with respect to $x_*$, we have
\[
\int_{\T^2} \mu(s,x_*)dx_*=\int_{\T^2}\left(-\frac{\triangle\sqrt\rho}{2\sqrt\rho}+\frac12|v|^2+f'(\rho)+V\right)dx_*.
\]
If, in addition, $\inf\rho>0$, by integrating by parts it follows that
\[
-\int_{\T^2}\frac{\triangle\sqrt\rho}{2\sqrt\rho}dx_*=-\int_{\T^2}\frac{|\nabla\sqrt\rho|^2}{2\rho}dx_*.
\]
Thus the averaged chemical potential can be expressed without second-order derivatives of $\rho$. Here we write this formulation as a lemma.

\begin{lem}\label{lem:phase2} 
Let $(\rho,v)$ be a GCP solution in the sense of Definition \ref{def:GCPsln}. If $\inf_{t,x}\rho>0$, then 
\begin{equation}\label{eq:phase2} 
\int_{\T^2} \mu(s,x_*)dx_* =\int_{\T^2} -\frac{|\nabla\sqrt{\rho}|^2}{2\rho}+\frac12|v|^2+f'(\rho)+V\,dx_*. 
\end{equation} 
\end{lem}

Now assume that $(\rho,v)$ is a GCP solution in the sense of Definition \ref{def:GCPsln}, with $\rho\geq\delta>0$, and that $v$ satisfies conditions \eqref{eq:condphase} and \eqref{eq:avrg0}. Then, by Definition \ref{def:phase}, we can construct a wave function $\psi=\sqrt{\rho}e^{iS}$ associated with $(\rho,v)$ in the sense of \eqref{eq:assoc}. Moreover, by the wave function lifting method established in the authors' previous work \cite{HaoMinE}, $\psi$ is shown to be a solution to the NLS-type equation \eqref{eq:NLS}, where the potential $S$ is also given by Definition \ref{def:phase}. The result in \cite{HaoMinE} is stated on $\R^2$, but with minor modifications it is also valid on $\T^2$.

\begin{prop}[Wave function lifting]\label{prop:lift2}
Let $\delta>0$.

\begin{enumerate}
\item Let $(\rho_0,v_0)(x)$ satisfy the assumptions of Theorem \ref{thm:glob2}. Then there exists a wave function $\psi_0\in H^2_x$, unique up to the choice of a constant phase, such that
\[
\rho_0=|\psi_0|^2,\qquad \rho_0v_0=\IM(\overline{\psi_0}\nabla\psi_0).
\]
Moreover, there exists a constant $C_0>0$, depending only on $\delta$ and on the fixed parameters of the system, such that
\[
I_0+C_0^{-1}E_0
\leq \frac12\|\nabla\psi_0\|_{H^1_x}^2
\leq I_0+C_0E_0.
\]

\item Let $(\rho,v)(t,x)$ be a GCP solution to the QHD system \eqref{eq:QHD_md} on $[0,T)\times\T^2$ in the sense of Definition \ref{def:GCPsln}, with strictly positive density $\inf_{t,x}\rho\geq\delta>0$.
Moreover, assume $v$ satisfies conditions \eqref{eq:condphase} and \eqref{eq:avrg0}. 
Then there exists a wave function $\psi\in L^\infty_tH^2_x$  
associated with $(\rho,v)$, namely
\[
\rho=|\psi|^2,\qquad \rho v=\IM(\bar\psi\nabla\psi),
\]
for almost every $t\in[0,T)$. Moreover, for almost every $t\in[0,T)$,
\begin{equation}\label{eq:H2_lift}
I(t)+C_0^{-1}E(t)
\leq \frac12\|\nabla\psi\|_{H^1_x}^2(t)
\leq I(t)+C_0E(t).
\end{equation}
Furthermore, $\psi$ is a solution to the NLS-type equation \eqref{eq:NLS} in the sense of distributions on $(0,T)\times\T^2$.
\end{enumerate}
\end{prop}

\section{Time-derivative computation of the main functionals}\label{sect:dt}

In this section, we derive the time-evolution identities for the functionals $E(t)$, $I(t)$ and the entropy $H(\rho)$ defined by \eqref{eq:en_md}, \eqref{eq:GCP} and \eqref{eq:entr}, in the framework of GCP solutions in the sense of Definition \ref{def:GCPsln} and the corresponding wave function $\psi$ provided by Proposition \ref{prop:lift2}.

Let $(\rho,v)$ be a GCP solution of the QHD system \eqref{eq:QHD_md} on the time interval $[0,T)$,  such that $\inf_{t,x}\rho\geq\delta>0$ and satisfies conditions \eqref{eq:condphase} and \eqref{eq:avrg0}. Assume also that $v$ satisfies the phase compatibility conditions \eqref{eq:condphase} and \eqref{eq:avrg0}. Then, by Proposition \ref{prop:lift2} and Remark \ref{rem:avrg0},  there exists a wave function $\psi\in L^\infty_t([0,T),H^2_x(\T^2))$ associated with $(\rho,v)$ in the sense of \eqref{eq:assoc}, such that $\psi$ solves the Cauchy problem for the Schr\"odinger--Langevin equation
\begin{equation}\label{eq:cauchy_NLS_1}
\begin{cases}
i\d_t\psi+\frac12\triangle\psi=f'(|\psi|^2)\psi+\frac{1}{\tau}S\psi+V\psi,\\
-\triangle V=|\psi|^2-M_0,\quad \psi(0)=\psi_0\in H^2_x(\T^2).
\end{cases}
\end{equation}
By the Madelung representation, we have
\[
|\nabla\psi|^2=|\nabla\sqrt\rho|^2+\rho |v|^2.
\]
We also recall from \eqref{eq:chem} that the chemical potential $\mu$ and the quantity $\sigma$ are given respectively by
\begin{equation}\label{eq:mu_psi}
\mu=-\frac12\RE\left(\frac{\triangle\psi}{\psi}\right)+f'(|\psi|^2)+V,\quad\sigma=-\frac12\IM\left(\frac{\triangle\psi}{\psi}\right).
\end{equation}

First, we have the following proposition concerning the time derivative of the main functionals $M(t)$, $E(t)$ and $I(t)$ introduced in Section 2.

\begin{prop}\label{prop:3.2}
Let $\psi\in L^\infty_t([0,T);H^2_x(\T^2))$ be a solution to the Cauchy problem \eqref{eq:cauchy_NLS_1} such that $\inf_{t,x}|\psi|>0$ and let $(\rho, v)$ be the hydrodynamic variables associated with $\psi$ in the sense of \eqref{eq:assoc}. Then the associated hydrodynamic variables satisfy the following properties:
\begin{itemize}
\item[(1)] the total mass is conserved
\begin{equation}\label{eq:cons_mass}
M(t)=\int_{\T^2}\rho(t,x) dx\equiv M_0,
\end{equation}
and the total energy defined by \eqref{eq:en_md} satisfies
\begin{equation}\label{eq:en_disp}
E(t)+\frac{1}{\tau}\int_0^t\int_{\T^2} \rho |v|^2 dxds=E(0)=E_0;
\end{equation}

\item[(2)] The function $t\mapsto I(t)$ is differentiable for almost every $t\in[0,T)$ and satisfies
\begin{equation}\label{eq:rl_I}
\frac{d}{dt} I(t)+\frac{1}{\tau}\int_{\T^2}\rho \sigma^2 dx=\int_{\T^2} \mu\d_tp(\rho)dx+\int_{\T^2}\rho\mu \d_tVdx-\frac{1}{\tau}\int_{\T^2}\rho |v|^2\mu dx,
\end{equation}
where $p(\rho)$ is the pressure given by \eqref{eq:pres}.

\end{itemize}
\end{prop}

\begin{proof}
Using equation \eqref{eq:cauchy_NLS_1}, one readily checks that the density $\rho=|\psi|^2$ satisfies the continuity equation
\[
\d_t\rho+\diver J=0,
\]
which directly implies the conservation of the total mass. To prove the energy relation \eqref{eq:en_disp}, we multiply the Schr\"odinger equation in \eqref{eq:cauchy_NLS_1}
by $-\triangle\bar\psi$, take the imaginary part, and integrate by parts to obtain
\[
\frac12\frac{d}{dt}\int_{\T^2}|\nabla\psi|^2dx=-\int_{\T^2}[f'(\rho)+\tau^{-1}S+V]\IM(\psi\triangle\bar\psi)dx,
\]
where
\[
\IM(\psi\triangle\bar\psi)=\diver\IM(\psi\nabla\bar\psi)=-\diver(\rho v)=\d_t\rho.
\]
Therefore
\begin{align*}
\frac{d}{dt}\int_{\T^2}\frac12|\nabla\psi|^2+f(\rho)dx=&\tau^{-1}\int_{\T^2} S\diver(\rho v) dx-\int_{\T^2}V\d_t\rho dx\\
=&-\tau^{-1}\int_{\T^2}\rho |v|^2 dx+\int_{\T^2} V\d_t(\triangle V-\mathcal{C}(x))dx\\
=&-\tau^{-1}\int_{\T^2}\rho |v|^2 dx-\frac{d}{dt}\int_{\T^2}\frac12|\nabla V|^2dx.
\end{align*}

In order to prove property (2), we first show that, for $\psi\in L^\infty_tH^2_x$ with $\inf_{t,x}|\psi|> 0$, the following identity holds: 
\begin{equation}\label{eq:prop_20_1}
\begin{aligned}
\frac12\frac{d}{dt}\int_{\T^2}\rho^{-1}[(\RE(\bar\psi\triangle\psi)-&2W\rho)^2+\IM(\bar\psi\triangle\psi)^2]dx\\
=&-\int_{\T^2}\IM[\triangle(W\bar\psi+\tau^{-1}S\bar\psi)\triangle\psi]dx\\
&+2\int_{\T^2} [\rho\d_t(W^2)+W^2\d_t \rho ]dx\\
&-2\int_{\T^2}\d_tW\RE(\bar\psi\triangle\psi)dx-4\int_{\T^2} W\RE(\d_t\bar\psi\triangle\psi)dx\\
&-4\int_{\T^2} \nabla W\cdot\RE(\nabla\bar\psi\d_t\psi)dx-2\int_{\T^2} \triangle W\RE(\bar\psi\d_t\psi)dx,
\end{aligned}
\end{equation}
where we have set $W=f'(\rho)+V$. However, to establish this identity rigorously, we need to introduce the standard sequence of mollifiers $\{\chi_\eps\}_{\eps>0}$ on $\T^2$ and let $g_\eps=g\ast \chi_\eps$ for any $g\in L^2_x(\T^2)$. We mollify equation \eqref{eq:cauchy_NLS_1}, multiply it by $\triangle^2\bar\psi_\eps$ and integrate by parts in the imaginary part, which gives
\begin{equation}\label{eq:3.16}
\frac12\frac{d}{dt}\int_{\T^2}|\triangle\psi_\eps|^2dx=-\int_{\T^2}\IM[\triangle(W\bar\psi+\tau^{-1}S\bar\psi)_\eps\triangle\psi_\eps]dx,
\end{equation}
where the function $S$ is defined by \eqref{eq:defS}. Using the algebra property of $H^2(\T^2)$, the elliptic regularity for $V$, and the regularity of $S$ given by Definition \ref{def:phase}, we have
\[
W\bar\psi+\tau^{-1}S\bar\psi\in L^\infty_tH^2_x.
\]
Now we expand the integrand in the left-hand side of \eqref{eq:3.16} as
\begin{align*}
|\triangle\psi_\eps|^2=|\psi_\eps|^{-2}|\bar\psi_\eps\triangle\psi_\eps|^2=\rho_\eps^{-1}[\RE(\bar\psi_\eps\triangle\psi_\eps)^2+\IM(\bar\psi_\eps\triangle\psi_\eps)^2],
\end{align*}
where $\rho_\eps=|\psi_\eps|^2$, which should not be confused with $\rho\ast\chi_\eps$. To match the chemical potential \eqref{eq:mu_psi}, we write
\begin{align*}
[\RE(\bar\psi_\eps\triangle\psi_\eps)]^2=&[\RE(\bar\psi_\eps\triangle\psi_\eps)-2W_\eps\rho_\eps]^2\\
&-4W_\eps^2\rho_\eps^2+4W_\eps\rho_\eps\RE(\bar\psi_\eps\triangle\psi_\eps).
\end{align*}
Thus we have
\begin{align*}
\frac12\frac{d}{dt}\int_{\T^2}|\triangle\psi_\eps|^2dx=&\frac12\frac{d}{dt}\int_{\T^2} \rho_\eps^{-1}[\RE(\bar\psi_\eps\triangle\psi_\eps)-2W_\eps\rho_\eps]^2dx\\
&+\frac12\frac{d}{dt}\int_{\T^2} \rho_\eps^{-1}\IM(\bar\psi_\eps\triangle\psi_\eps)^2dx\\
&+2\frac{d}{dt}\int_{\T^2} [W_\eps\RE(\bar\psi_\eps\triangle\psi_\eps)-W_\eps^2\rho_\eps]dx.
\end{align*}
Now we compute the time derivative of the last integral,
\begin{align*}
\frac{d}{dt}\int_{\T^2} W_\eps\RE(\bar\psi_\eps\triangle\psi_\eps)dx=&\int_{\T^2}\d_tW_\eps\RE(\bar\psi_\eps\triangle\psi_\eps)dx+\int_{\T^2} W_\eps\RE(\d_t\bar\psi_\eps\triangle\psi_\eps)dx\\
&+\int_{\T^2} W_\eps\RE(\bar\psi_\eps\triangle\d_t\psi_\eps)dx\\
=&\int_{\T^2} \d_tW_\eps\RE(\bar\psi_\eps\triangle\psi_\eps)dx+2\int_{\T^2} W_\eps\RE(\d_t\bar\psi_\eps\triangle\psi_\eps)dx\\
&+2\int_{\T^2} \nabla W_\eps\RE(\nabla\bar\psi_\eps\d_t\psi_\eps)dx+\int_{\T^2} \triangle W_\eps\RE(\bar\psi_\eps\d_t\psi_\eps)dx,
\end{align*}
and
\[
\frac{d}{dt}\int_{\T^2} W_\eps^2\rho_\eps dx=\int_{\T^2} [\d_t (W_\eps^2)\rho_\eps+W_\eps^2\d_t \rho_\eps] dx.
\]
Since $\psi\in L^\infty_tH^2_x$ with $\inf_{t,x}|\psi|>0$ and, by equation \eqref{eq:cauchy_NLS_1}, $\d_t\psi\in L^\infty_tL^2_x$, the corresponding mollified quantities converge strongly in the required spaces. We may therefore pass to the limit as $\eps\to0$ and obtain \eqref{eq:prop_20_1}.

In order to reduce \eqref{eq:prop_20_1} to \eqref{eq:rl_I}, it follows from \eqref{eq:mu_psi} that
\[
\RE(\bar\psi\triangle\psi)=-2\rho\mu+2\rho W,\quad \IM(\bar\psi\triangle\psi)=-2\rho\sigma,
\]
and therefore 
\[
\frac12\frac{d}{dt}\int_{\T^2}\rho^{-1}[(\RE(\bar\psi\triangle\psi)-2\rho W)^2+\IM(\bar\psi\triangle\psi)^2]dx=4\frac{d}{dt}I(t).
\]
We further compute that
\begin{align*}
-\IM[\triangle(W\bar\psi+\tau^{-1}S\bar\psi)\triangle\psi]=&-[\triangle W+\tau^{-1}\diver v]\IM(\bar\psi\triangle\psi)\\
&-2[\nabla W+\tau^{-1}v]\cdot\IM(\nabla\bar\psi\triangle\psi),
\end{align*}
where we use $\nabla S=v$, and 
\[
-2\d_t W\RE(\bar\psi\triangle\psi)=4\d_t W\rho\mu-2\rho\d_t (W^2).
\]
By equation \eqref{eq:cauchy_NLS_1}, we also have
\[
-4W\RE(\d_t\bar\psi\triangle\psi)=4W(W+\tau^{-1}S)\IM(\bar\psi\triangle\psi),
\]
\[
-4\nabla W \cdot\RE(\nabla\bar\psi\d_t\psi) =2\nabla W \cdot \IM(\nabla\bar\psi\triangle\psi) +4\nabla W\cdot(W+\tau^{-1}S)\IM(\bar\psi\nabla\psi),
\]
and
\[
-2\triangle W\RE(\bar\psi\d_t\psi)=\triangle W\IM(\bar\psi\triangle \psi).
\]
Notice that
\[
\IM(\bar\psi\triangle\psi)=\diver\IM(\bar\psi\nabla\psi)=\diver(\rho v)=-\d_t\rho,
\]
then by summarising the identities above and integrating by parts, we obtain
\begin{align*}
4\frac{d}{dt}I(t)=&4\int_{\T^2}\rho\mu\d_tWdx-\tau^{-1}\int_{\T^2}\diver v\diver(\rho v)dx\\
&-4\tau^{-1}\int_{\T^2}\rho |v|^2Wdx-2\tau^{-1}\int_{\T^2} v\cdot\IM(\nabla\bar\psi\triangle\psi)dx.
\end{align*}
Last, by using
\begin{align*}
\IM(\nabla\bar\psi\triangle\psi)=&\rho\IM\left(\frac{\nabla\bar\psi}{\bar\psi}\frac{\triangle\psi}{\psi}\right)\\
=&\rho\left[\IM\left(\frac{\nabla\bar\psi}{\bar\psi}\right)\RE\left(\frac{\triangle\psi}{\psi}\right)+\RE\left(\frac{\nabla\bar\psi}{\bar\psi}\right)\IM\left(\frac{\triangle\psi}{\psi}\right)\right]\\
=&\rho[v(2\mu-2W)+\frac12\rho^{-2}\nabla\rho\diver(\rho v)],
\end{align*}
it follows that
\[
\frac{d}{dt}I(t)=\int_{\T^2}\rho\mu\d_tWdx-\tau^{-1}\int_{\T^2}\rho |v|^2\mu dx-\frac{1}{4\tau}\int_{\T^2}\frac{[\diver(\rho v)]^2}{\rho}dx,
\]
which proves \eqref{eq:rl_I} by using \eqref{eq:pres} and
\[
\frac{[\diver(\rho v)]^2}{\rho}=\frac{(\d_t\rho)^2}{\rho}=4\rho\sigma^2.
\]

\end{proof}

We also need the following computation of the time derivative of the physical entropy 
\[
H(\rho)=\int_{\T^2}\rho\log\left(\frac{\rho}{M_0}\right)dx.
\]
\begin{prop}\label{prop:dtentr}
Let $(\rho,v)$ be a GCP solution of \eqref{eq:QHD_md} such that $\inf_{t,x}\rho>0$. Then the time derivative of $H(\rho)$ is given by
\begin{equation}\label{eq:dtH}
\begin{aligned}
\frac{d}{dt}H(\rho)=&-\tau\int_{\T^2}\rho|\nabla^2\log\sqrt\rho|^2 dx-4\tau\int_{\T^2}p'(\rho)|\nabla\sqrt\rho|^2dx\\
&-\tau\frac{d}{dt}\int_{\T^2}\log\rho\d_t\rho dx+4\tau\int_{\T^2}\rho\sigma^2 dx\\
&+\tau\int_{\T^2}(\rho v\otimes v):\nabla^2\log\rho dx-\tau\int_{\T^2}(\rho-M_0)^2 dx.
\end{aligned}
\end{equation}
As a consequence, we have the estimate
\begin{equation}\label{eq:ineq_entr}
\begin{aligned}
\frac{d}{dt}&[H(\rho)+\tau\int_{\T^2}\log\rho\d_t\rho dx]+\frac{\tau}{2}\int_{\T^2}\rho|\nabla^2\log\sqrt\rho|^2dx\\
+&4\tau\int_{\T^2} p'(\rho)|\nabla\sqrt\rho|^2dx+\tau\int_{\T^2}(\rho-M_0)^2dx\leq4\tau\int_{\T^2}\rho\sigma^2dx+2\tau\int_{\T^2} \rho |v|^4dx.
\end{aligned}
\end{equation}
\end{prop}

\begin{proof}
To rigorously establish the computation, we again need to use the standard mollifiers $\{\chi_\eps\}_{\eps>0}$. Let $\rho_{\eps}=\rho\ast\chi_\eps$, and consider the entropy $H(\rho_{\eps})$ of the mollified density. By mollifying equation \eqref{eq:QHD_md} and using \eqref{eq:bohm} to write the quantum term in logarithm form, the time derivative of $H(\rho_{\eps})$ is computed as
\begin{align*}
\frac{d}{dt}H(\rho_{\eps})=&\int_{\T^2}\log\rho_{\eps}\d_t\rho_{\eps} dx=-\int_{\T^2}\log\rho_{\eps}\diver(\rho v)_\eps dx\\
=&-\frac{\tau}{2}\int_{\T^2}\log\rho_{\eps}\diver\cdot\diver(\rho\nabla^2\log\sqrt\rho)_\eps dx+\tau\int_{\T^2}\log\rho_{\eps}\triangle p(\rho)_\eps dx\\
&+\tau\int_{\T^2}\log\rho_{\eps}\d_t\diver(\rho v)_\eps dx+\tau\int_{\T^2}\log\rho_{\eps}\diver\cdot\diver(\rho v\otimes v)_\eps dx\\
&+\tau\int_{\T^2}\log\rho_{\eps}\diver(\rho \nabla V)_\eps dx\\
=&-\tau\int_{\T^2}(\rho\nabla^2\log\sqrt\rho)_\eps :\nabla^2\log\sqrt\rho_{\eps}dx-\tau\int_{\T^2} \nabla p(\rho)_\eps\cdot\nabla\log \rho_{\eps}dx\\
&-\tau\int_{\T^2}\log\rho_{\eps}\d_t^2\rho_{\eps} dx+\tau\int_{\T^2}(\rho v\otimes v)_\eps:\nabla^2\log\rho_{\eps} dx\\
&-\tau\int_{\T^2}\frac{\nabla\rho_{\eps}}{\rho_{\eps}}\cdot(\rho \nabla V)_\eps dx.
\end{align*}
Here we can further write 
\[
-\tau\int_{\T^2}\log\rho_{\eps}\d_t^2\rho_{\eps} dx=-\tau\frac{d}{dt}\int_{\T^2}\log\rho_{\eps}\d_t\rho_{\eps} dx+\tau\int_{\T^2}\d_t\log\rho_{\eps}\d_t\rho_{\eps} dx.
\]
We now justify the passage to the limit as $\eps\to0$. By the definition of GCP solutions and the strict positivity assumption $\inf_{t,x}\rho>0$, all the quantities appearing in the right-hand side of the regularized identity are well defined as integrable functions. In particular, the continuity equation gives $\d_t\rho=-\diver(\rho v)$ in the sense of distributions, while the bounds in Definition \ref{def:GCPsln} ensure that the products obtained after mollification are locally uniformly integrable on $[0,T)\times\T^2$. Therefore, the regularized quantities converge to the corresponding non-regularized ones in the sense of distributions, and the uniform bounds allow us to pass to the limit in the integral identity. Letting $\eps\to0$, we obtain
\begin{align*}
\frac{d}{dt}H(\rho)=&-\tau\int_{\T^2}\rho|\nabla^2\log\sqrt\rho|^2 dx-\tau\int_{\T^2} \nabla p(\rho)\cdot\nabla\log \rho dx\\
&-\tau\frac{d}{dt}\int_{\T^2}\log\rho\d_t\rho dx+\tau\int_{\T^2}\d_t\log\rho\d_t\rho dx\\
&+\tau\int_{\T^2}(\rho v\otimes v):\nabla^2\log\rho dx-\tau\int_{\T^2}\nabla\rho\cdot\nabla V dx\\
=&-\tau\int_{\T^2}\rho|\nabla^2\log\sqrt\rho|^2 dx-4\tau\int_{\T^2} p'(\rho)|\nabla\sqrt\rho|^2dx\\
&-\tau\frac{d}{dt}\int_{\T^2}\log\rho\d_t\rho dx+4\tau\int_{\T^2}\rho\sigma^2 dx\\
&+\tau\int_{\T^2}(\rho v\otimes v):\nabla^2\log\rho dx+\tau\int_{\T^2}\rho\triangle V dx.
\end{align*}
Last, by using the Poisson equation 
\[
-\triangle V=\rho-M_0
\]
and $M_0=\int_{\T^2}\rho dx$, we obtain \eqref{eq:dtH}.

The inequality \eqref{eq:ineq_entr} follows from the estimate
\begin{align*}
\tau\left|\int_{\T^2}(\rho v\otimes v):\nabla^2\log\rho dx\right|\leq& 2\tau\|\rho^\frac14 v\|_{L^4_x}^2\|\sqrt\rho\nabla^2\log\sqrt\rho\|_{L^2_x}\\
\leq &2\tau\int_{\T^2} \rho |v|^4dx+\frac{\tau}{2}\int_{\T^2}\rho|\nabla^2\log\sqrt\rho|^2dx.
\end{align*}
\end{proof}

Finally, we establish the following lemma, which shows that the norms $\|\nabla^2\sqrt{\rho}\|_{L^2_x}$ and $\|\nabla\rho^{\frac14}\|_{L^4_x}$ are controlled by the dissipation term $\int_{\T^2}\rho|\nabla^2\log\sqrt{\rho}|^2dx$ appearing in \eqref{eq:dtH}.

\begin{lem}\label{lem:H^2log}
Assume that $\rho\in H^2_x(\T^2)$ and $\rho>0$. Then
\begin{equation}\label{eq:logH2}
\int_{\T^2}\frac14|\triangle\sqrt{\rho}|^2dx+\int_{\T^2}\frac12|\nabla^2\sqrt{\rho}|^2+\frac{|\nabla\sqrt\rho|^4}{4\rho}dx\leq  \int_{\T^2}\rho|\nabla^2\log\sqrt{\rho}|^2dx.
\end{equation}
\end{lem}

\begin{proof}
Direct computation gives
\[
\sqrt{\rho}\nabla^2\log\sqrt{\rho}=\sqrt{\rho}\nabla\left(\frac{\nabla\sqrt\rho}{\sqrt\rho}\right)=\nabla^2\sqrt{\rho}-\frac{\nabla\sqrt\rho\otimes\nabla\sqrt\rho}{\sqrt\rho},
\]
and
\[
\sqrt{\rho}\triangle\log\sqrt{\rho}=\sqrt{\rho}\diver\left(\frac{\nabla\sqrt\rho}{\sqrt\rho}\right)=\triangle\sqrt{\rho}-\frac{|\nabla\sqrt\rho|^2}{\sqrt\rho}.
\]
Thus we have
\begin{align*}
D_1=\int_{\T^2}\rho|\nabla^2\log\sqrt{\rho}|^2dx=&\int_{\T^2}|\nabla^2\sqrt{\rho}|^2+\frac{|\nabla\sqrt\rho|^4}{\rho}dx-2\int_{\T^2}\nabla^2\sqrt{\rho}:\frac{\nabla\sqrt\rho\otimes\nabla\sqrt\rho}{\sqrt\rho}dx\\
=&A_1+B-I_1
\end{align*}
and
\begin{align*}
D_2=\int_{\T^2}\rho(\triangle\log\sqrt{\rho})^2dx=&\int_{\T^2}(\triangle\sqrt{\rho})^2+\frac{|\nabla\sqrt\rho|^4}{\rho}dx-2\int_{\T^2}\triangle\sqrt{\rho}\frac{|\nabla\sqrt\rho|^2}{\sqrt\rho}dx\\
=&A_2+B-I_2.
\end{align*}
By integrating by parts, it follows that
\begin{align*}
I_2=&-4\int_{\T^2}\nabla^2\sqrt{\rho}:\frac{\nabla\sqrt\rho\otimes\nabla\sqrt\rho}{\sqrt\rho}dx+2\int_{\T^2}\frac{|\nabla\sqrt\rho|^4}{\rho}dx\\
=&-2I_1+2B.
\end{align*}
Therefore, we have
\[
2D_1+D_2=2(A_1+B-I_1)+A_2-B+2I_1=2A_1+A_2+B.
\]
Finally, since $D_2\leq 2D_1$ in two space dimensions, we obtain \eqref{eq:logH2}.

\end{proof}

\section{Global well-posedness of GCP solutions on $\T^2$}\label{sect:exist2d}

The global well-posedness of GCP solutions to system \eqref{eq:QHD}, namely Theorem \ref{thm:glob2}, is proved by establishing the well-posedness of $H^2_x$ solutions to \eqref{eq:NLS} and by using the wave function lifting result in Proposition \ref{prop:lift2} to connect \eqref{eq:QHD} with \eqref{eq:NLS}.

In the two-dimensional case, we need the following logarithmic Sobolev-type inequality, which was proved in \cite{BG} on $\R^2$. Here we present an improved version of this inequality on $\T^2$.

\begin{lem}\label{lem:logsob}
For any $u\in H^2(\T^2)$ with $\int_{\T^2}u dx=0$, we have 
\begin{equation}\label{eq:logsob}
\|u\|_{L^\infty_x(\T^2)}\le C_0\|\nabla u\|_{L^2_x(\T^2)} \left[1+\left|\log \frac{\|\triangle u\|_{L^2_x(\T^2)}}{\|\nabla u\|_{L^2_x(\T^2)}}\right|^\frac12\right].
\end{equation}
\end{lem}

\begin{proof}
The case $\nabla u=0$ is trivial, since, combining with $\int_{\T^2}u dx=0$, it implies $u=0$, and \eqref{eq:logsob} trivially holds. In the following proof, we assume $\nabla u$ is not identically $0$.

Let $\{\hat u_{j,k}\}_{(j,k)\in\Z^2}$ be the Fourier coefficients of $u$ on $\T^2$. Then, by the assumption 
\[
\hat u_{0,0}=\int_{\T^2}u dx=0.
\]
By the Plancherel theorem, it follows that 
\begin{equation}\label{eq:plan1}
\|u\|_{L^\infty(\T^2)}\leq |\{\hat u_{j,k}\}|_{l^1(\Z^2)},\quad \|\nabla u\|_{L^2(\T^2)}=\left|\{\sqrt{j^2+k^2}\,\hat u_{j,k}\}\right|_{l^2(\Z^2)}
\end{equation}
and
\begin{equation}\label{eq:plan2}
\|\triangle u\|_{L^2(\T^2)}=\left|\{(j^2+k^2)\hat u_{j,k}\}\right|_{l^2(\Z^2)},
\end{equation}
where $|\cdot|_{l^p(\Z^2)}$ is the $l^p(\Z^2)$ norm of a sequence given by
\[
|\{w_{j,k}\}|_{l^p(\Z^2)}=\left(\sum_{(j,k)\in\Z^2}|w_{j,k}|^p\right)^\frac1p.
\]
For arbitrary $R>0$, by the Cauchy-Schwarz inequality and \eqref{eq:plan1}, \eqref{eq:plan2}, we have
\begin{align*}
|\{\hat u_{j,k}\}|_{l^1(\Z^2)}=&\sum_{|(j,k)|<R,\ (j,k)\ne(0,0)}|\hat u_{j,k}|+\sum_{|(j,k)|\ge R}|\hat u_{j,k}|\\
=&\sum_{|(j,k)|<R,\ (j,k)\ne(0,0)}\frac{1}{\sqrt{j^2+k^2}}|\sqrt{j^2+k^2}\,\hat u_{j,k}|\\
&+\sum_{|(j,k)|\ge R}\frac{1}{j^2+k^2}|(j^2+k^2)\hat u_{j,k}|\\
\leq & C_0|\sqrt{j^2+k^2}\,\hat u_{j,k}|_{l^2(\Z^2)}\left(\sum_{|(j,k)|<R,\ (j,k)\ne(0,0)}\frac{1}{j^2+k^2}\right)^\frac12\\
&+C_0|(j^2+k^2)\hat u_{j,k}|_{l^2(\Z^2)}\left(\sum_{|(j,k)|\geq R}\frac{1}{(j^2+k^2)^2}\right)^\frac12\\
\leq & C_0\|\nabla u\|_{L^2(\T^2)} |\log R|^\frac12+ C_0\|\triangle u\|_{L^2(\T^2)} R^{-1}.
\end{align*}
Now consider the function
\[
g(R)=a\log R+b R^{-2},\quad a,b>0,\;R\geq 1,
\]
then by letting $a=\|\nabla u\|_{L^2(\T^2)}^2$ and $b=\|\triangle u\|_{L^2(\T^2)}^2$, the previous inequalities imply
\[
\|u\|_{L^\infty(\T^2)}\leq C_0\,g(R)^\frac12.
\]
It is straightforward to check from
\[
g'(R)=R^{-3}\left(aR^2-2b\right)
\]
that $g(R)$ has a unique minimum at $R_0$, where
\[
R_0=\sqrt{\frac{2b}{a}}=\frac{\sqrt{2}\|\triangle u\|_{L^2(\T^2)}}{\|\nabla u\|_{L^2(\T^2)}}.
\]
If $\sqrt2\,\|\triangle u\|_{L^2(\T^2)}\leq \|\nabla u\|_{L^2(\T^2)}$, namely $R_0\leq 1$, we choose $R=1$, which implies
\[
\|u\|_{L^\infty(\T^2)}\leq C_0\|\nabla u\|_{L^2(\T^2)}.
\]
If $\sqrt{2}\,\|\triangle u\|_{L^2(\T^2)}> \|\nabla u\|_{L^2(\T^2)}$, then we can choose $R=R_0>1$, which yields \eqref{eq:logsob}.
\end{proof}

As the first step to establish the well-posedness result,we consider the Cauchy problem associated with the NLS-type equation \eqref{eq:NLS},
\begin{equation}\label{eq:cauchy_NLS}
\begin{cases}
i\d_t\psi+\frac12\triangle\psi=f'(|\psi|^2)\psi+\frac{1}{\tau}S\psi+V\psi,\\
-\triangle V=|\psi|^2-\mathcal{C}(x),\quad \psi(0)=\psi_0,
\end{cases}
\end{equation}
for initial data $\psi_0\in H^2_x(\T^2)$, where $S$ denotes the phase of $\psi$, equivalently given by Definition \ref{def:phase} for the associated hydrodynamic variables. We solve for $\psi$ such that $\inf_{t,x}|\psi|\geq\delta^\frac12>0$. By Proposition \ref{prop:lift2} of the wave function lifting, $\psi$ is associated with the hydrodynamic variables
\[
\rho=|\psi|^2,\quad v=\IM\left(\frac{\nabla\psi}{\psi}\right)
\]
and the well-posedness result for \eqref{eq:cauchy_NLS}, together with the lifting and polar factorization arguments, yields the corresponding well-posedness result for the QHD system \eqref{eq:QHD_md}.

Using the logarithmic embedding inequality \eqref{eq:logsob}, we have the following estimates on the $L^\infty_x(\T^2)$ norm of $\psi$ and $\rho-M_0$, which play an essential role in this section.

\begin{lem}\label{lem:Linfty_2d}
Let $\psi\in H^2_x(\T^2)$, $\rho=|\psi|^2$ and $M_0=\int_{\T^2}\rho dx$. Then the following estimates hold:
\begin{equation}\label{eq:Linfty1}
\|\psi\|_{L^\infty_x}\leq M_0^\frac12+g_1(\|\nabla \psi\|_{L^2_x},\|\nabla\psi\|_{H^1_x})
\end{equation}
and
\begin{equation}\label{eq:Linfty2}
\|\rho-M_0\|_{L^\infty_x}\leq g_2(\|\nabla \psi\|_{L^2_x},\|\nabla\psi\|_{H^1_x}),
\end{equation}
where, for $u,w>0$,
\[
g_1(u,w)=C_0u(1+|\log (w/u)|^\frac12)
\]
and
\[
g_2(u,w)=C_0 g_1(u,w)\cdot u[1+|\log (M_0^\frac12+g_1(u,w))|+|\log (w/u)|^\frac12].
\]
\end{lem}

\begin{proof}
Inequality \eqref{eq:Linfty1} is a direct consequence of \eqref{eq:logsob},
\begin{align*}
\|\psi\|_{L^\infty_x}\leq & \int_{\T^2}|\psi|dx+C\|\nabla \psi\|_{L^2_x} \left[1+\left|\log \frac{\|\triangle \psi\|_{L^2_x}}{\|\nabla \psi\|_{L^2_x}}\right|^\frac12\right]\\
\leq & \left(\int_{\T^2}|\psi|^2dx\right)^\frac12+C\|\nabla \psi\|_{L^2_x} \left[1+\left|\log \frac{\|\triangle \psi\|_{L^2_x}}{\|\nabla \psi\|_{L^2_x}}\right|^\frac12\right]\\
= & M_0^\frac12+g_1(\|\nabla \psi\|_{L^2_x},\|\nabla\psi\|_{H^1_x}).
\end{align*}
Also by \eqref{eq:logsob}, we have
\begin{equation}\label{eq:temp1}
\|\rho-M_0\|_{L^\infty_x}\leq C_0\|\nabla \rho\|_{L^2_x}(1+|\log (\|\triangle \rho\|_{L^2_x}/\|\nabla \rho\|_{L^2_x})|^\frac12).
\end{equation}
We further have
\[
\|\nabla \rho\|_{L^2_x}\leq 2\|\psi\|_{L^\infty_x}\|\nabla\psi\|_{L^2_x},
\]
and 
\begin{align*}
\|\triangle \rho\|_{L^2_x}\leq & \|\psi\|_{L^\infty_x}\|\triangle \psi\|_{L^2_x}+\|\nabla \psi\|_{L^4_x}^2\\
\leq & \|\psi\|_{L^\infty_x}\|\triangle \psi\|_{L^2_x}+C_0\|\nabla \psi\|_{L^2_x}\|\triangle \psi\|_{L^2_x}.
\end{align*}
Substituting the estimates of $\|\nabla \rho\|_{L^2_x}$, $\|\triangle \rho\|_{L^2_x}$ into \eqref{eq:temp1} and using \eqref{eq:Linfty1}, we obtain \eqref{eq:Linfty2}.
\end{proof}

Let $(\rho_0,J_0)$ be the initial data in Theorem \ref{thm:glob2}, and set $v_0=J_0/\rho_0$. We define $\psi_0\in H^2_x(\T^2)$ to be the lifted wave function associated with $(\rho_0,v_0)$, as provided by Proposition \ref{prop:lift2}. Then $\psi_0$ satisfies $\inf_x|\psi_0|\geq\delta^\frac12>0$. Now we can prove the local existence of $H^2$ solutions to \eqref{eq:cauchy_NLS} with initial data $\psi_0$ by a standard Picard iteration scheme.

\begin{prop}
There exists $C_*(M_0,\delta)>0$, with $M_0=\|\psi_0\|_{L^2_x}^2$, such that if
\begin{equation}\label{eq:ini_1}
\|\nabla\psi_0\|_{H^1_x} \leq C_*(M_0,\delta),
\end{equation}
then the Cauchy problem \eqref{eq:cauchy_NLS} has a unique local-in-time solution $\psi\in \mathcal{C}([0,T_*);H^2_x(\T^2))$, for some $T_*>0$ depending on $\tau$, $\delta$ and $\|\psi_0\|_{H^2_x}$, with
\[
\inf_{(t,x)\in[0,T_*)\times\T^2}|\psi(t,x)|\geq \delta^{1/2}.
\]
\end{prop}

\begin{proof}
For nonlinear Schr\"odinger equations without the Langevin term $\tau^{-1}S\psi$, local well-posedness is well established in the literature, see for example \cite{KK} and references therein. Therefore, in this proof we only present the technical details concerning the Langevin term $\tau^{-1}S\psi$.

The proof follows a standard Picard iteration scheme. We construct an approximating sequence $\{\psi_m\}$ by solving the following linear Schr\"odinger equation for $\psi_m$:
\begin{equation}\label{eq:QLS_m}
i\d_t\psi_m+\frac12\triangle\psi_m=f'(|\psi_{m-1}|^2)\psi_m+\frac{1}{\tau}S_{m-1}\psi_m+V_{m-1}\psi_m,
\end{equation}
with $\psi_{m-1}$, $S_{m-1}$ and $V_{m-1}$ given in the previous step with $\psi_{m-1}\in \mathcal{C}([0,T_*);H^2_x(\T^2))$ and $\inf_{t,x}|\psi_{m-1}|\geq \delta^\frac12$. The Duhamel formula for \eqref{eq:QLS_m} is written as
\begin{equation}\label{eq:duhamel}
\begin{aligned}
\psi_m=e^{\frac{i}{2}t\triangle}\psi_0&-i\int_0^te^{\frac{i}{2}(t-s)\triangle}[f'(|\psi_{m-1}|^2)+V_{m-1}]\psi_m(s)ds\\
&-\frac{i}{\tau}\int_0^te^{\frac{i}{2}(t-s)\triangle}S_{m-1}\psi_m(s)ds,
\end{aligned}
\end{equation}
and the electric potential $V_{m-1}$ is provided by the Poisson equation, 
\[
-\triangle V_{m-1}=|\psi_{m-1}|^2-M_0
\]
with $\int_{\T^2} V_{m-1}dx=0$. For the phase function, we first set $S_0=0$. For $m\geq 1$, we want to give a suitable definition of $S_m$ by adapting the idea of Definition \ref{def:phase} for $\psi_m$ as follows. As before, we define the velocity and the linearized chemical potential as
\begin{equation}\label{eq:mu_m}
v_m=\IM\left(\frac{\nabla\psi_m}{\psi_m}\right),\quad \mu_m=-\frac12\RE\left(\frac{\triangle\psi_m}{\psi_m}\right)+f'(|\psi_{m-1}|^2)+V_{m-1},
\end{equation}
and we will see below that this definition is consistent by the positivity of $|\psi_m|$. By \eqref{eq:QLS_m} the equation of $v_m$ is given by
\begin{equation}\label{eq:dtv_m}
\operatorname{curl}v_m=0,\quad \d_t v_m=-\nabla\left(\mu_m+\frac{1}{\tau}S_{m-1}\right),\quad v_m(0)=v_0.
\end{equation}
Thus following the idea of \eqref{eq:gradS}, we define $S_m$ to be the solution to the space-time gradient equation 
\begin{equation}\label{eq:gradS_m}
\begin{cases}
\nabla S_m =v_m,\quad\d_t S_m=-\mu_m-\frac{1}{\tau}S_{m-1}, \\
S_m(0,0)=S_*\in [0,2\pi).
\end{cases}
\end{equation}
Together with the zero-average condition in \eqref{eq:cond_1} (which is also satisfied by $v_m$ following the dynamics), the irrotationality condition \eqref{eq:dtv_m} guarantees the solvability of \eqref{eq:gradS_m} with periodic $S_m$. The phase function $S_m$ can be explicitly written as
\begin{equation}\label{eq:defS_m}
\begin{aligned}
S_m(t,x)=&S_*+\int_{\T^2}\int_{l(0,x_*)}v_0(y)\cdot d\vec{l}(y)dx_*\\
&-\int_0^t \int_{\T^2}(\mu_m+\tau^{-1}S_{m-1})(s,x_*)dx_*ds+\int_{\T^2}\int_{l(x_*,x)}v_m(t,y)\cdot d\vec{l}(y)dx_*.
\end{aligned}
\end{equation}

We claim that the sequence $\{\psi_m\}$ satisfies the following properties on $[0,T_*)\times \T^2$:
\begin{itemize}
\item[(1)] boundedness: $\|\psi_m\|_{L^\infty_tH^2_x}\leq 2\|\psi_0\|_{H^2_x}$ and $\|S_{m}\|_{L^\infty_tH^2_x}\leq C(\delta,\|\psi_0\|_{H^2_x})$;
\item[(2)] positivity: $\inf_{t,x}|\psi_m|\geq \delta^\frac12$.
\end{itemize}
By our assumption, $\psi_0$ and $S_0$ obviously satisfy the properties (1) and (2). In this paper we will only focus on the estimates related to the phase function $S_m$ and the positivity property (2), and the remaining analysis follows the standard argument for nonlinear Schr\"odinger equations.

First, we determine the bound of the initial data $\|\nabla\psi_0\|_{H^1_x}\leq C_*(M_0,\delta)$, such that property (1) implies property (2). Let us define the density $\rho_m=|\psi_m|^2$ and the velocity
\[
v_m=\IM\left(\frac{\nabla\psi_m}{\psi_m}\right).
\]
Then, from \eqref{eq:QLS_m}, we obtain 
\[
\d_t\rho_m+\diver (\rho_m v_m)=0.
\]
As a consequence, the total mass (and also the average density, since we assume $|\T^2|=1$) $\int_{\T^2}\rho_m dx\equiv M_0$ is conserved. Then, by \eqref{eq:Linfty2}, we have
\[
|\rho_m-M_0|\leq g_2(\|\nabla\psi_m\|_{L^2_x},\|\triangle\psi_m\|_{L^2_x}),
\]
and we can choose $C_*(M_0,\delta)$ small such that if $\|\nabla\psi_m\|_{H^1_x(\T^2)}\leq C_*(M_0,\delta)$ holds, we have
\[
|\rho_m-M_0|\leq g_2(\|\nabla\psi_m\|_{L^2_x},\|\triangle\psi_m\|_{L^2_x}) \leq M_0-\delta,
\]
which also implies
\[
|\psi_m|\geq \delta^\frac12.
\]

For the boundedness of $\|\nabla\psi_m\|_{L^\infty_tH^1_x}$, we can directly consider the $H^2_x$ norm of $\psi_m$, as we have shown $\|\psi_m\|_{L^2_x}^2\equiv M_0$. The last integral in \eqref{eq:duhamel} can be estimated as
\begin{align*}
\|\int_0^te^{\frac{i}{2}(t-s)\triangle}S_{m-1}\psi_m(s)ds\|_{L^\infty_tH^2_x}\leq&  T_* \|S_{m-1}\psi_m \|_{L^\infty_tH^2_x}\\
\leq & C_0T_* \|S_{m-1}\|_{L^\infty_tH^2_x}\|\psi_m \|_{L^\infty_tH^2_x}.
\end{align*}
Thus by the assumption of induction, we have
\[
\frac{1}{\tau}\|\int_0^te^{\frac{i}{2}(t-s)\triangle}S_{m-1}\psi_m(s)ds\|_{L^\infty_tH^2_x}\leq C(\delta,\|\psi_0\|_{H^2_x})\frac{T_*}{\tau} \|\psi_m \|_{L^\infty_tH^2_x},
\]
and we choose $T_*>0$ small such that $C(\delta,\|\psi_0\|_{H^2_x})\frac{T_*}{\tau}\leq \frac14$. The other terms in \eqref{eq:duhamel} are treated similarly, thus we obtain
\[
\|\psi_m\|_{L_t^\infty H^2_x}\leq \|\psi_0\|_{H^2_x}+\frac12\|\psi_m\|_{L_t^\infty H^2_x},
\]
namely the first part of property (1). 

Last, we prove the second part of property (1), namely $\|S_m\|_{L^\infty_tH^2_x}\leq C(\delta,\|\psi_0\|_{H^2_x})$. By $\nabla S_m=v_m$, we only need to control $\|S_m\|_{L^\infty_tL^2_x}$ and $\|\nabla v_m\|_{L^\infty_tL^2_x}$.
From \eqref{eq:defS_m} we obtain
\begin{align*}
|S_m(t,x)|\leq& S_*+ \int_{\T^2}|v_0|dx_*+\int_0^t\left|\int_{\T^2}\mu_mdx_*\right|ds\\
&+\int_0^t\int_{\T^2}\tau^{-1}|S_{m-1}|dx_*ds+\int_{\T^2}|v_m|dx_*.
\end{align*}
Recalling that $v_m=\IM(\frac{\nabla\psi_m}{\psi_m})$, by the bound of $\psi_m$ and (2), we choose $C(\delta,\|\psi_0\|_{H^2_x})$ such that
\[
\int_{\T^2}|v_m|dx_*\leq 2\delta^{-\frac12}\int_{\T^2} |\nabla\psi_m|dx_*\leq \frac14C(\delta,\|\psi_0\|_{H^2_x}).
\]
The same bound holds for $\int_{\T^2}|v_0|dx_*$. By the assumption of induction we have $\|S_{m-1}\|_{L^\infty_tL^2_x}\leq C(\delta,\|\psi_0\|_{H^2_x})$, therefore
\[
\int_0^t\int_{\T^2}\tau^{-1}|S_{m-1}|dx_*ds\leq \frac14C(\delta,\|\psi_0\|_{H^2_x})
\]
if we choose $T_*\leq \frac{\tau}{4}$. Last, analogous to \eqref{eq:phase2} we can choose $T_*>0$ small such that
\begin{align*}
\int_0^t\left|\int_{\T^2}\mu_mdx_*\right|ds\leq & 2\delta^{-1}\int_0^t\int_{\T^2}|\nabla\psi_m|^2dx_*ds\\
&+\int_0^t\int_{\T^2}|f'(|\psi_{m-1}|^2)|dx_*ds\leq \frac{1}{4}C(\delta,\|\psi_0\|_{H^2_x}),
\end{align*}
where in the last inequality we use $f\in C^2([0,\infty))$ and 
\[
|f'(|\psi_{m-1}|^2)|\leq C(\|f\|_{C^2},\|\psi_{m-1}\|_{L^\infty_{t,x}})\leq C(\|f\|_{C^2},\|\psi_{m-1}\|_{L^\infty_{t}H^2_x}).
\]
Thus we prove $\|S_m\|_{L^\infty_{t,x}}\leq C(\delta,\|\psi_0\|_{H^2_x})$. On the other hand, again by the definition of $v_m$, we have
\[
\|\nabla v_m\|_{L_t^\infty L^2_x}\leq \delta^{-\frac12}\|\triangle \psi_m\|_{L_t^\infty L^2_x}+\delta^{-1}\|\nabla\psi_m\|_{L_t^\infty L^4_x}^2,
\]
where by the embedding inequality,
\[
\|\nabla\psi_m\|_{L_t^\infty L^4_x}^2\leq \|\psi_m\|_{L_t^\infty H^2_x}^2\leq 4\|\psi_0\|_{H^2_x}^2.
\]
Thus we finish the proof of property (1).

By an analogous argument, applied to the difference of two equations of the form \eqref{eq:duhamel}, we can prove that $\{(\psi_m,S_m)\}$ is a contractive sequence. Since the estimates are standard and almost repetitive, we will not give the details.

\begin{lem}\label{lem:contr}
There exists $0<q<1$ such that the sequence $\{(\psi_m,S_m)\}$ satisfies
\[
\|\psi_m-\psi_{m-1}\|_{L^\infty_tH^2_x}+\|S_{m}-S_{m-1}\|_{L^\infty_tH^2_x}\leq C(\delta,\|\psi_0\|_{H^2_x})q^{m-1}.
\]
\end{lem}

Lemma \ref{lem:contr} implies $\{(\psi_m,S_m)\}$ converges strongly to $(\psi,S)$ in $L^\infty([0,T_*],H^2(\T^2))$. Moreover by the convergence and the uniform lower bound of $|\psi_m|$, we have
\[
v_m=\IM\left(\frac{\nabla\psi_m}{\psi_m}\right)\to v=\IM\left(\frac{\nabla\psi}{\psi}\right),
\]
\[
-\triangle V_m=|\psi_m|^2-\mathcal{C}(x)\to-\triangle V=|\psi|^2-\mathcal{C}(x), 
\]
and
\[
\mu_m=-\frac12\RE\left(\frac{\triangle\psi_m}{\psi_m}\right)+f'(|\psi_{m-1}|^2)+V_{m-1}\to -\frac12\RE\left(\frac{\triangle\psi}{\psi}\right)+f'(|\psi|^2)+V=\mu.
\] 
Also by writing the equation of $S_m$ in \eqref{eq:gradS_m} as
\[
\d_tS_m=-\frac{1}{\tau}S_m-\mu_m-\frac{1}{\tau}(S_{m-1}-S_m),
\]
we have
\begin{align*}
S_m(t,x)=&e^{-\frac{t}{\tau}}\left(S_*+\int_{\T^2}\int_{l(0,x_*)}v_0(y)\cdot d\vec{l}dx_*\right)\\
&-\int_0^t \int_{\T^2} e^{\frac{s-t}{\tau}}[\mu_m+\tau^{-1}(S_{m-1}-S_m)](s,x_*)dx_*ds\\
&+\int_{\T^2}\int_{l(x_*,x)}v_m(t,y)\cdot d\vec{l}(y)dx_*,
\end{align*}
then by passing to the limit we see $S$ satisfies \eqref{eq:defS}. Thus we prove $(\psi,S)$ is a solution to the Cauchy problem \eqref{eq:cauchy_NLS} on $[0,T_*)\times\T^2$, and the continuity in time follows directly from the Duhamel formula.

The proof of uniqueness follows from a standard argument for the Schr\"odinger equation, by considering the difference equation, similarly to the contraction estimate in Lemma \ref{lem:contr}.

\end{proof}

Let $T^*>0$ be the maximal time of existence for which $\psi$ solves the Cauchy problem \eqref{eq:cauchy_NLS} and belongs to the class
\begin{equation}\label{eq:spaceX}
\begin{aligned}
X(T^*)=\{\psi\in L^\infty([0,T^*);H^2_x(\T^2)); \|\nabla &\psi\|_{L^\infty_tH^1_x}\leq C_*(M_0,\delta),\\
&\||\psi|^2-M_0\|_{L^\infty_{t,x}}\leq M_0-\delta\},
\end{aligned}
\end{equation}
and let $(\rho,v)$ be the hydrodynamic variables associated with $\psi$ in the sense of \eqref{eq:assoc}. Here  $M_0=\int_{\T^2}\rho dx$, and $C_*(M_0,\delta)$ is the bound appearing in \eqref{eq:ini_1}. By the choice of $C_*(M_0,\delta)$ in the proof of the previous Proposition of local existence, the bound  
\[
\|\nabla\psi\|_{L^\infty_tH^1_x}\leq C_*(M_0,\delta)
\]
automatically implies 
\[
\left||\psi|^2-M_0\right|\leq M_0-\delta,
\]
but we still emphasize this property in the definition of the space $X(T^*)$, since it plays an essential role in our later arguments. Now we will show that, by choosing the initial data suitably small, we can extend the time of existence $T^*$ to infinity. The main tools are some apriori estimates based on the functionals $E(t)$, $I(t)$ and the entropy $H(\rho)(t)$, whose time derivatives are computed in Section \ref{sect:dt}. Moreover, these estimates should be established uniformly with respect to the relaxation time $\tau>0$, thus they allow us to rigorously justify the relaxation-time limit.

\begin{prop}\label{prop:bdI}
Let $\psi\in L^\infty(0,T^*;H^2_x(\T^2))$ be the solution considered above. Then for almost every $0\leq t<T^*$, the functional $I(t)$ satisfies 
\begin{equation}\label{eq:ineqI}
\begin{aligned}
\frac{d}{dt} & I(t)+\frac{1}{4\tau}\int_{\T^2}\rho\sigma^2dx+\frac{1}{2\tau}\int_{\T^2}\rho |v|^4 dx \\
\leq & C_0g_3(M_0,E_0,\delta)\tau \int_{\T^2}\rho\mu^2dx+C_0\delta^{-1}E_0[M_0+|f'(\delta)|+I(t)]\frac{1}{\tau}\int_{\T^2} \rho |v|^2dx,
\end{aligned}
\end{equation}
where
\begin{equation}\label{eq:g_2d}
g_3(M_0,E_0,\delta)=M_0+M_0(M_0-\delta)^{4n-2}+M_0^\frac43E_0^\frac23.
\end{equation}
\end{prop}

\begin{proof}
Recalling \eqref{eq:rl_I} in Proposition \ref{prop:3.2}, we have
\[
\frac{d}{dt} I(t)+\frac{1}{\tau}\int_{\T^2}\rho \sigma^2 dx=\int_{\T^2} \mu\d_tp(\rho)dx+\int_{\T^2}\rho\mu\d_tV dx-\frac{1}{\tau}\int_{\T^2}\rho |v|^2\mu dx.
\]
The right-hand side can be estimated as follows. First we have
\[
\int_{\T^2} \mu\d_tp(\rho)dx\leq 2\|p'(\rho)\|_{L^\infty_x}\|\sqrt\rho\mu\|_{L^2_x}\|\d_t\sqrt\rho\|_{L^2_x}.
\]
Recalling that
\[
p'(\rho)=(2n-1)\rho(\rho-M_0)^{2n-2},
\]
and using the definition of the space $X(T^*)$ in \eqref{eq:spaceX}, it follows that
\[
\|p'(\rho)\|_{L^\infty_x}\leq 2(2n-1)M_0(M_0-\delta)^{2n-2}.
\]
Thus we obtain
\begin{align*}
\int_{\T^2} \mu\d_tp(\rho)dx\leq & 2(2n-1)M_0(M_0-\delta)^{2n-2}\|\sqrt\rho\mu\|_{L^2_x}\|\d_t\sqrt\rho\|_{L^2_x}\\
\leq & C_0M_0^2(M_0-\delta)^{4n-4}\tau\int_{\T^2}\rho\mu^2dx+\frac{1}{4\tau}\int_{\T^2}\rho\sigma^2dx.
\end{align*}

For the integral $\int_{\T^2} \rho\mu\d_tVdx$, we have
\[
\int_{\T^2} \rho\mu\d_tVdx\leq 2\|\sqrt\rho\|_{L^6_x}\|\sqrt\rho\mu\|_{L^2_x}\|\d_tV\|_{L^3_x}.
\]
By the Poisson equation of $V$, we have $-\triangle\d_tV=\d_t\rho$, then by embedding inequalities it follows that
\[
\|\d_tV\|_{L^3_x}\leq C_0\|\d_t\rho\|_{L^\frac32_x}=2C_0\|\sqrt\rho\|_{L^6_x}\|\d_t\sqrt\rho\|_{L^2_x},
\]
and
\[
\|\sqrt\rho\|_{L^6_x}\leq C_0\|\sqrt\rho\|_{L^2_x}^\frac23\|\nabla\sqrt\rho\|_{L^2_x}^{\frac13}+C_0\|\sqrt\rho\|_{L^2_x}\leq C_0(M_0^\frac12+M_0^\frac13E_0^\frac16).
\]
Thus we obtain
\begin{align*}
\int_{\T^2} \rho\mu\d_tVdx\leq & C_0(M_0+M_0^\frac23E_0^\frac13)\|\sqrt\rho\mu\|_{L^2_x}\|\sqrt\rho\sigma\|_{L^2_x}\\
\leq & C_0(M_0^2+M_0^\frac43E_0^\frac23)\tau \|\sqrt\rho\mu\|_{L^2_x}^2+\frac{1}{4\tau}\int_{\T^2}\rho\sigma^2dx.
\end{align*}

For the last integral in the right-hand side of \eqref{eq:rl_I}, by recalling
\[
\rho\mu=-\frac14\triangle\rho+\frac12|\nabla\sqrt\rho|^2+\frac12\rho |v|^2+f'(\rho)\rho+\rho V,
\]
we write 
\begin{align*}
-\tau^{-1}\int_{\T^2}\rho |v|^2\mu dx=&\frac{1}{4\tau}\int_{\T^2} |v|^2\triangle\rho dx-\frac{1}{2\tau}\int_{\T^2} |v|^2|\nabla\sqrt\rho|^2dx\\
&-\frac{1}{2\tau}\int_{\T^2} \rho |v|^4 dx-\tau^{-1}\int_{\T^2} (f'(\rho)+V)\rho |v|^2dx.
\end{align*}
By integrating by parts, it follows that
\begin{align*}
\frac{1}{4\tau}\int_{\T^2} |v|^2\triangle\rho dx=-\frac{1}{\tau}\int_{\T^2}   \nabla\sqrt\rho\cdot(\sqrt\rho\nabla v)\cdot v dx\\
=-\frac{1}{\tau}\int_{\T^2}  v \cdot\nabla (\sqrt\rho v)\cdot\nabla\sqrt\rho dx+\frac{1}{\tau}\int_{\T^2}  |v|^2|\nabla\sqrt\rho|^2 dx
\end{align*}
Recall that $v=\nabla S$ is irrotational, which implies
\[
\operatorname{curl}(\sqrt\rho v)=\nabla\sqrt\rho\times v,
\]
then by Helmholtz decomposition, we have
\begin{align*}
\|\nabla (\sqrt\rho v)\|_{L^2_x}\leq & C_0\|\diver (\sqrt\rho v)\|_{L^2_x}+C_0\|\nabla\sqrt\rho\times v\|_{L^2_x}\\
\leq & C_0 \|\sqrt\rho \sigma\|_{L^2_x}+C_0\||\nabla\sqrt\rho| |v|\|_{L^2_x}.
\end{align*}
Therefore we obtain
\[
\frac{1}{4\tau}\int_{\T^2} |v|^2\triangle\rho dx\leq \frac{1}{8\tau}\int_{\T^2}  \rho\sigma^2 dx+\frac{C_0}{\tau}\int_{\T^2}  |v|^2|\nabla\sqrt\rho|^2 dx.
\]
Now we need to estimate the integral of $|v|^2|\nabla\sqrt\rho|^2$. Recall that $\inf_{t,x}\rho\geq\delta$, then by the Gagliardo-Nirenberg inequality,  we have
\begin{align*}
\delta\int_{\T^2}  |v|^2|\nabla\sqrt\rho|^2 dx\leq & \|\sqrt\rho v\|_{L^4_x}^2\|\nabla\psi\|_{L^4_x}^2\\
\leq & C_0\|\sqrt\rho v\|_{L^2_x}\|\nabla(\sqrt\rho v)\|_{L^2_x}\|\nabla\psi\|_{L^2_x}\|\triangle\psi\|_{L^2_x}\\
&+C_0\|\sqrt\rho v\|_{L^2_x}^2\|\nabla\psi\|_{L^2_x}^2\\
\leq & C_0\|\sqrt\rho v\|_{L^2_x}\left(\|\sqrt\rho \sigma\|_{L^2_x}+\||\nabla\sqrt\rho| |v|\|_{L^2_x}\right)\|\nabla\psi\|_{L^2_x}\|\triangle\psi\|_{L^2_x}\\
&+C_0\|\sqrt\rho v\|_{L^2_x}^2\|\nabla\psi\|_{L^2_x}^2,
\end{align*}
which implies
\begin{align*}
\int_{\T^2}  |v|^2|\nabla\sqrt\rho|^2 dx\leq & C_0\delta^{-1}\|\sqrt\rho v\|_{L^2_x}\|\sqrt\rho \sigma\|_{L^2_x}\|\nabla\psi\|_{L^2_x}\|\triangle\psi\|_{L^2_x}\\
&+C_0\delta^{-1}\|\sqrt\rho v\|_{L^2_x}^2\|\nabla\psi\|_{L^2_x}^2\|\triangle\psi\|_{L^2_x}^2+C_0\delta^{-1}\|\sqrt\rho v\|_{L^2_x}^2\|\nabla\psi\|_{L^2_x}^2.
\end{align*}
Again by Proposition \ref{prop:lift2} we have
\[
\frac12\|\triangle\psi\|_{L^2_x}^2\leq C_0\left(I(t)+E(t)\right).
\]
Thus we obtain
\begin{align*}
\frac{1}{\tau}\int_{\T^2}  |v|^2|\nabla\sqrt\rho|^2dx\leq \frac{1}{8\tau}\int_{\T^2}  \rho\sigma^2 dx+C_0\delta^{-1}E_0[1+I(t)]\tau^{-1}\int_{\T^2} \rho |v|^2dx.
\end{align*}
Last, again by the bound $|\rho-M_0|\leq M_0-\delta$, we have
\begin{align*}
\frac{1}{\tau}\int_{\T^2} (f'(\rho)+V)\rho |v|^2dx\leq&  \|f'(\rho)+V\|_{L^\infty_x}\frac{1}{\tau}\int_{\T^2} \rho |v|^2dx\\
\leq & C_0(|f'(\delta)|+M_0)\frac{1}{\tau}\int_{\T^2} \rho |v|^2dx.
\end{align*}
Summarizing the inequalities above, we obtain
\begin{align*}
\frac{d}{dt} & I(t)+\frac{1}{4\tau}\int_{\T^2}\rho\sigma^2dx+\frac{1}{2\tau}\int_{\T^2}\rho |v|^4 dx \\
\leq & C_0g_3(M_0,E_0,\delta)\tau \int_{\T^2}\rho\mu^2dx+C_0\delta^{-1}E_0[M_0+|f'(\delta)|+I(t)]\frac{1}{\tau}\int_{\T^2} \rho |v|^2dx,
\end{align*}
which finishes the proof.
\end{proof}

As a consequence of Proposition \ref{prop:3.2}, Proposition \ref{prop:dtentr} and Proposition \ref{prop:bdI}, we have the following decay estimate of a combination functional.

\begin{prop}\label{prop:F_2d}
Let
\begin{equation}\label{eq:c1_2d}
c_1=\frac{1}{4C_0nM_0E_0}\min\left\{g_3(M_0,E_0,\delta)^{-1},[\delta^{-1}E_0(M_0+|f'(\delta)|)]^{-1}\right\},
\end{equation}
where $g_3$ is given in Proposition \ref{prop:bdI}, and let us define the combination functional
\begin{equation}\label{eq:F}
F(t)=H(\rho)+E(t)+c_1I(t).
\end{equation}
Then on the time interval $[0,T^*)$ of existence for $\psi\in X(T^*)$, there exist $c_2,\tau^*>0$ such that for $0<\tau<\tau^*$, we have
\begin{equation}\label{eq:decay}
F(t)\leq F(0)\exp\left(\frac{C_0}{\tau}\int_0^t\int_{\T^2}\rho |v|^2dxds- c_2\tau t\right).
\end{equation}
\end{prop}

\begin{proof}
By using \eqref{eq:en_disp}, \eqref{eq:ineq_entr}, \eqref{eq:ineqI} and our choice of $c_1$, the time derivative of $F(t)$ satisfies
\begin{equation}\label{eq:dtF_1}
\begin{aligned}
\frac{d}{dt}[F(t)+&\tau\int_{\T^2}\log\rho\d_t\rho dx]+\frac{1}{\tau}\int_{\T^2}\rho |v|^2 dx+\left(\frac{c_1}{\tau}-4\tau\right)\int_{\T^2}\rho\sigma^2dx\\
&+\left(\frac{c_1}{2\tau}-\frac{2\tau}{\delta}\right)\int_{\T^2} \rho |v|^4 dx+\frac{\tau}{2}\int_{\T^2}\rho|\nabla^2\log\sqrt\rho|^2dx\\
&+4\tau\int_{\T^2} p'(\rho)|\nabla\sqrt\rho|^2dx+\tau\int_{\T^2}(\rho-M_0)^2dx\\
\leq & \frac{\tau}{4C_0nM_0E_0}\int_{\T^2}\rho\mu^2dx+c_1I(t)\frac{C_0}{\tau}\int_{\T^2} \rho |v|^2dx.
\end{aligned}
\end{equation}
To prove the exponential decay \eqref{eq:decay}, we need to control the right-hand side and $F(t)$ itself by the dissipative terms.

From \eqref{eq:chem} we can write
\[
2\sqrt\rho\mu=-\triangle\sqrt\rho+\sqrt\rho |v|^2+2(f'(\rho)+V)\sqrt\rho,
\]
which implies
\[
2\int_{\T^2}\rho\mu^2 dx\leq \int_{\T^2}(\triangle\sqrt\rho)^2dx+\int_{\T^2} \rho |v|^4 dx+4\int_{\T^2} (f'(\rho)^2+V^2)\rho dx.
\]
For the last integral, recalling that $f(\rho)=(2n)^{-1}(\rho-M_0)^{2n}$, we have
\[
\int_{\T^2} f'(\rho)^2\rho dx=\int_{\T^2} \rho(\rho-M_0)^{4n-2} dx\leq 2M_0 \int_{\T^2} (\rho-M_0)^{4n-2} dx.
\]
On the other hand, by 
\[
p'(\rho)=\rho f''(\rho)=(2n-1)\rho(\rho-M_0)^{2n-2},
\]
we have
\begin{align*}
4\int_{\T^2} p'(\rho)|\nabla\sqrt\rho|^2dx=&\int_{\T^2} (2n-1)(\rho-M_0)^{2n-2}|\nabla\rho|^2dx\\
=&(2n-1)n^{-2}\int_{\T^2}|\nabla(\rho-M_0)^n|^2dx.
\end{align*}
Then by combining it with the bound of the internal energy and using Gagliardo-Nirenberg inequality, it follows that
\begin{align*}
\int_{\T^2} f'(\rho)^2\rho dx\leq & C_0 M_0 \left[\int_{\T^2}(\rho-M_0)^{2n}\right]^{2-\frac{1}{n}-\beta_n}\left[\int_{\T^2}|\nabla(\rho-M_0)^n|^2\right]^{\beta_n},
\end{align*}
where
\[
\beta_n=1-\frac{1}{n}.
\]
Combining it with the energy dissipation \eqref{eq:en_disp}, we obtain
\[
\int_{\T^2} f'(\rho)^2\rho dx\leq C_0nM_0E_0\int_{\T^2} p'(\rho)|\nabla\sqrt\rho|^2dx.
\]
Also, by the Poisson equation of $V$ and Poincar\'e inequality, we have
\[
\int_{\T^2}\rho V^2dx\leq \|\rho\|_{L^\infty_x}\int_{\T^2} V^2dx\leq C_0M_0\int_{\T^2} (\rho-M_0)^2dx.
\]
Thus by our choice of $c_1$, we obtain
\begin{align*}
 &\frac{1}{4C_0nM_0E_0}\int_{\T^2}\rho\mu^2 dx\\
\leq & \frac{1}{4C_0nM_0E_0}\left[\int_{\T^2}(\triangle\sqrt\rho)^2dx+\int_{\T^2} \rho |v|^4 dx\right.\\
&\hspace{2.5cm}\left.+4\int_{\T^2} f'(\rho)^2\rho dx+4\int_{\T^2}\rho V^2dx\right]\\
\leq& \frac14\int_{\T^2}(\triangle\sqrt\rho)^2dx+\frac{1}{4}\int_{\T^2} \rho |v|^4 dx\\
&+2\int_{\T^2} p'(\rho)|\nabla\sqrt\rho|^2dx+\frac12\int_{\T^2}(\rho-M_0)^2dx.
\end{align*}
Using \eqref{eq:logH2}, we can rewrite \eqref{eq:dtF_1} as
\begin{equation}\label{eq:dtF_2}
\begin{aligned}
\frac{d}{dt}[F(t)+&\tau\int_{\T^2}\log\rho \d_t\rho  dx]+\frac{1}{2\tau}\int_{\T^2}\rho |v|^2 dx+\frac{\tau}{4} \int_{\T^2}\rho\mu^2 dx\\
&+\left(\frac{c_1}{2\tau}-4\tau\right)\int_{\T^2}\rho \sigma ^2dx+\left(\frac{c_1}{2\tau}-\frac{2\tau}{\delta}-\frac{\tau}{4}\right)\int_{\T^2} \rho  |v| ^4 dx\\
&+\frac{\tau}{4}\int_{\T^2}|\nabla^2\sqrt\rho|^2dx
+\frac{8\tau}{3}\int_{\T^2}|\nabla\rho^\frac14|^4dx
+2\tau\int_{\T^2} p'(\rho )|\nabla\sqrt\rho|^2dx\\
&+\frac{\tau}{2}\int_{\T^2}(\rho-M_0)^2dx\leq  c_1I(t)\frac{C_0}{\tau}\int_{\T^2} \rho |v|^2dx,
\end{aligned}
\end{equation}
where in the left-hand side we have
\[
\frac{\tau}{4}\int_{\T^2}\rho\mu^2 dx+\left(\frac{c_1}{2\tau}-4\tau\right)\int_{\T^2}\rho\sigma^2dx\geq \frac{\tau}{2} I(t)
\]
when $\tau\leq \tau^*\leq \frac{c_1^\frac12}{4}$.

It remains to control $H(\rho)+E(t)$ by the dissipative terms in \eqref{eq:dtF_2}. We recall the following technical lemma.

\begin{lem}[Logarithmic Sobolev inequality \cite{DGJ}]\label{lem:logineq}
Let $u\in H^1_x(\T^2)$ and $\overline{u^2}=|\T^2|^{-1}\int_{\T^2} u^2 dx$. Then
\begin{equation}\label{eq:logineq}
\int_{\T^2} u^2\log(u^2/\overline{u^2})dx\leq C_0\int_{\T^2} |\nabla u|^2dx.
\end{equation}
\end{lem}

By letting $u=\sqrt\rho$ in Lemma \ref{lem:logineq} and applying again Poincar\'e inequality, we obtain
\[
H(\rho)\leq C_0\int_{\T^2} |\nabla \sqrt\rho|^2dx\leq C_0\int_{\T^2} |\nabla^2 \sqrt\rho|^2dx.
\]
For the total energy $E(t)$, the quantum part $\int_{\T^2}|\nabla\sqrt\rho|^2dx$ also follows from the inequality above. On the other hand, using the embedding inequality and the interpolation, the internal energy can be bounded by
\begin{align*}
\int_{\T^2} f(\rho)dx=&\frac{1}{2n}\int_{\T^2}(\rho-M_0)^{2n}dx\\
\leq &\frac{C_0}{n}\int_{\T^2}|\nabla(\rho-M_0)^n|^2dx+\frac{C_0}{n}\left|\int_{\T^2}(\rho-M_0)^n dx\right|^2\\
\leq & C_0 \int_{\T^2} p'(\rho )|\nabla\sqrt\rho |^2dx+C(M_0,E_0)\int_{\T^2}(\rho -M_0)^2dx.
\end{align*}
We also have the bound for the electric potential:
\[
\int_{\T^2}|\nabla V|^2dx\leq C_0\int_{\T^2}(\rho-M_0)^2dx.
\]
Thus there exists $c_2>0$ such that
\begin{align*}
\frac{1}{2\tau}\int_{\T^2}\rho |v|^2 dx+\frac{\tau}{4}\int_{\T^2}|\nabla^2\sqrt\rho |^2dx&+2\tau\int_{\T^2} p'(\rho )|\nabla\sqrt\rho |^2dx\\
&+\frac{\tau}{2}\int_{\T^2}(\rho -M_0)^2dx\geq c_2\tau[H(\rho)+E(t)].
\end{align*}
Then by letting $\tau$ small such that
\[
\frac{c_1}{2\tau}-\frac{2\tau}{\delta}-\frac{\tau}{4}\geq 0,
\]
namely $\tau\leq \left(\frac{2c_1\delta}{8+\delta}\right)^\frac12$, it follows from \eqref{eq:dtF_2} that
\begin{equation}\label{eq:dtF_3}
\frac{d}{dt}[F(t)+\tau\int_{\T^2}\log\rho\d_t\rho dx]+c_2\tau F(t)\leq \frac{C_0}{\tau}F(t)\int_{\T^2} \rho |v|^2dx.
\end{equation}
Finally, by integrating by parts,
\begin{align*}
\tau\left|\int_{\T^2}\log\rho\,\d_t\rho\,dx\right|
=\tau\left|\int_{\T^2}\log\rho\,\diver J\,dx\right|
=2\tau\left|\int_{\T^2} \nabla \sqrt\rho\cdot (\sqrt\rho v)\,dx\right|
\leq 2\tau E(t).
\end{align*}
Then, we have for $\tau< \frac14$ that
\begin{equation}\label{eq:equivF}
\frac12F(t)\leq F(t)+\tau\int_{\T^2}\log\rho\d_t\rho dx\leq \frac32 F(t).
\end{equation}
Thus combining \eqref{eq:dtF_3}, \eqref{eq:equivF} and using the Gronwall inequality, we finish the proof of \eqref{eq:decay}.

\end{proof}

As a consequence of Proposition \ref{prop:bdI}, we can extend the local existence result for the Cauchy problem \eqref{eq:cauchy_NLS} to a global one, which proves the global well-posedness of the Schr\"odinger--Langevin equation stated in Theorem \ref{thm:NLS}.

\begin{proof}[Proof of Theorem \ref{thm:NLS}]
By the decay estimate \eqref{eq:decay}, on the time interval $[0,T^*)$ of existence, we have
\[
F(t)\leq F(0)\exp\left(\frac{C_0}{\tau}\int_0^t\int_{\T^2}\rho |v|^2dxds\right),
\]
where the combined functional $F(t)$ is given by \eqref{eq:F}. The energy dissipation \eqref{eq:en_disp} implies
\[
\frac{1}{\tau}\int_0^t\int_{\T^2}\rho |v|^2dxds\leq E_0.
\]
Then it follows from \eqref{eq:H2_lift} that
\begin{align*}
\|\nabla\psi(t,\cdot)\|_{H^1_x}\leq & \sqrt{2}\,[I(t)+C_0 E(t)]^\frac12\\
\leq & [2c_1^{-1} F(t)]^\frac12\leq \sqrt{2}\,c_1^{-\frac12} F(0)^\frac12 e^{C_0E_0}.
\end{align*}
Therefore, for initial data for which $F(0)$ is suitably small, we have
\[
\|\nabla\psi\|_{L^\infty_tH^1_x}\leq C_*(M_0,\delta)
\]
for all $t\in[0,T^*)$, which indicates that we can extend the local well-posedness result to a global one, and the decay estimate \eqref{eq:decay} always holds.

Moreover, we can check the dependence of the initial functional norm $(E_0,F(0))$ on $M_0$. By the $L^\infty_{x}$ bound \eqref{eq:Linfty2} and our argument of the well-posedness result, the bound of $(E_0,F(0))$ needs to guarantee that
\[
g_2(E_0^\frac12,c_1^{-\frac12} F(0)^\frac12 e^{C_0E_0})\leq M_0-\delta,
\]
where $g_2$ is given in Lemma \ref{lem:Linfty_2d}. A sufficient condition for this inequality is
\begin{equation}\label{eq:dep_2d}
C_0 E_0\left[1+|\log c_1^{-\frac12}|+|\log F(0)|\right]\leq M_0-\delta.
\end{equation}
By the definition of $c_1$ in Proposition \ref{prop:F_2d}, we have
\[
c_1^{-\frac12}\leq C_0E_0\delta^{-\frac12}\sqrt{n}(M_0-\delta)^{2n}.
\]
For fixed $\delta>0$ and $n\in\N$, the left-hand side of \eqref{eq:dep_2d} depends logarithmically on $M_0$, while the right-hand side is linear in $M_0$. For example, we can choose an $\eps_0>0$ small and independent of the norms of initial data, such that condition \eqref{eq:cond_2} in Theorem \ref{thm:glob2}: 
\[
e^{E_0}\cdot(1+ I_0)\leq \eps_0\frac{e^{M_0-\delta}}{(M_0-\delta)^{2n}},
\]
then \eqref{eq:dep_2d} holds. Using the equivalence between the hydrodynamic energy functionals and the Sobolev norms of $\psi$,
\[
E_0 \sim \|\nabla\psi_0\|_{L^2}^2,
\qquad
I_0 \sim \|\nabla\psi_0\|_{H^1}^2,
\]
we obtain the following equivalent smallness condition in terms of the Sobolev norms of $\psi_0$:
\[
e^{C_0\|\nabla\psi_0\|_{L^2_x}^2}
\left(1+\|\nabla\psi_0\|_{H^1_x}^2\right)
\le
\eps_0\frac{e^{M_0-\delta}}{(M_0-\delta)^{2 n}},
\]
for a suitable constant $C_0>0$, which is precisely the smallness assumption in Theorem \ref{thm:NLS}. 
Moreover, when $M_0$ is large, this condition allows both the initial total energy $E_0$ and the initial value $F(0)$ to be large.
\end{proof}

We are now ready to prove the hydrodynamic well-posedness result, Theorem \ref{thm:glob2}.

\begin{proof}[Proof of Theorem \ref{thm:glob2}]
Let the initial data $(\rho_0,J_0)$ be as in Theorem \ref{thm:glob2}. Using the lifting construction developed in Proposition \ref{prop:lift2}, we define a wave function $\psi_0\in H^2_x(\T^2)$ associated with the initial hydrodynamic data $(\rho_0,v_0)$ with $v_0=J_0/\rho_0$. Moreover,
\[
\inf_x|\psi_0|=\inf_x\sqrt\rho_0\geq \delta^\frac12,
\]
and
\[
\|\nabla\psi_0\|_{H^1_x} \leq C_*(M_0,\delta).
\]
Thus by setting $\psi_0$ as initial data and applying Theorem \ref{thm:NLS}, we solve the Cauchy problem for the Schr\"odinger--Langevin equation \eqref{eq:cauchy_NLS} to obtain a global solution $\psi\in X(T)$ for any $0<T<\infty$, where $X(t)$ is defined by \eqref{eq:spaceX}. 

Now we define the hydrodynamic variables associated with $\psi$ as 
\[
\rho=|\psi|^2,\quad J=\rho v=\IM(\bar\psi\nabla\psi),
\]
and prove $(\rho, J)$ is a weak solution to \eqref{eq:QHD} in the sense of Definition \ref{def:GCPsln}. Using \eqref{eq:cauchy_NLS}, direct computation shows
\[
\d_t\rho=2\RE(\bar\psi\d_t\psi)=-\IM(\bar\psi\triangle\psi)=-\diver\IM(\bar\psi\nabla\psi)=-\diver J,
\]
namely the continuity equation holds. To prove the momentum equation, we again need to use the standard mollifiers $\{\chi_\eps\}_{\eps>0}$ and define 
\[
\psi_\eps=\psi\ast\chi_\eps,\quad J_\eps=\IM(\bar\psi_\eps\nabla\psi_\eps).
\]
Again we define the potential $W=f'(\rho)+V$. Then $J_\eps$ satisfies the equation
\begin{align*}
\d_tJ_\eps=&\IM(\d_t\bar\psi_\eps\nabla\psi_\eps)+\IM(\bar\psi_\eps\nabla\d_t\psi_\eps)\\
=& -\frac12\RE(\nabla\bar\psi_\eps\triangle\psi_\eps)+\RE[(W\bar\psi+\tau^{-1}S\bar\psi)_\eps\nabla\psi_\eps]\\
&+\frac12\RE(\bar\psi_\eps\nabla\triangle\psi_\eps)-\RE[\bar\psi_\eps\nabla(W\psi+\tau^{-1}S\psi)_\eps].
\end{align*}
Take $\zeta\in\mathcal C^\infty_0([0, T)\times\T^2;\R^2)$ to be an arbitrary test function. Then, by integrating by parts, we have
\begin{align*}
\int_0^T\int_{\T^2} J_\eps \cdot\d_t\zeta dxdt=&-\int_{\T^2} J_\eps(0)\cdot\zeta(0)dx-\frac12\int_0^T\int_{\T^2}\RE(\nabla\bar\psi_\eps\otimes\nabla\psi_\eps):\nabla\zeta dxdt\\
&+\frac12\int_0^T\int_{\T^2}\RE(\bar\psi_\eps\nabla^2\psi_\eps):\nabla \zeta\, dxdt\\
&-\int_0^T\int_{\T^2}\zeta\cdot\RE[(W\bar\psi+\tau^{-1}S\bar\psi)_\eps\nabla\psi_\eps] dxdt\\
&+\int_0^T\int_{\T^2} \zeta\cdot\RE[\bar\psi_\eps\nabla(W\psi+\tau^{-1}S\psi)_\eps]dxdt.
\end{align*}
Since
\[
\psi\in L^\infty_tH^2_x,
\quad
S,V,W\in L^\infty_tH^2_x,
\]
all mollified quantities converge strongly in the corresponding Sobolev spaces; hence, we may pass to the limit \(\eps\to0\). It follows that
\begin{align*}
\int_0^T\int_{\T^2} J\cdot\d_t\zeta dxdt=&-\int_{\T^2} J_0\cdot\zeta(0)dx-\frac12\int_0^T\int_{\T^2}\RE(\nabla\bar\psi\otimes\nabla\psi):\nabla\zeta\, dxdt\\
&+\frac12\int_0^T\int_{\T^2}\RE(\bar\psi\nabla^2\psi):\nabla\zeta\, dxdt\\
&-\int_0^T\int_{\T^2} \zeta\cdot\RE[(W\bar\psi+\tau^{-1}S\bar\psi)\nabla\psi] dxdt\\
&+\int_0^T\int_{\T^2} \zeta\cdot\RE[\bar\psi\nabla(W\psi+\tau^{-1}S\psi)]dxdt.
\end{align*}
We further compute
\[
\RE(\bar\psi\nabla^2\psi)=\frac12\nabla^2\rho-\RE(\nabla\bar\psi\otimes\nabla\psi),
\]
and
\begin{align*}
-\RE[(W\bar\psi+\tau^{-1}S\bar\psi)\nabla\psi]+\RE[&\bar\psi\nabla(W\psi+\tau^{-1}S\psi)]\\
=&\rho\nabla [f'(\rho)+V]+\tau^{-1}\rho\nabla S\\
=&\nabla p(\rho)+\rho\nabla V+\tau^{-1}J,
\end{align*}
where in the last identity we have used 
\[
p(\rho)=f'(\rho)\rho-f(\rho),\quad \nabla S=v.
\]
Thus we obtain
\begin{align*}
\int_0^T\int_{\T^2} J \cdot\d_t\zeta dxdt=&-\int_{\T^2} J_0\cdot\zeta(0)dx-\int_0^T\int_{\T^2}\RE(\nabla\bar\psi\otimes\nabla\psi):\nabla\zeta\, dxdt\\
&+\frac14\int_0^T\int_{\T^2}\nabla^2\rho:\nabla \zeta\, dxdt-\int_0^T\int_{\T^2} \zeta \cdot(\nabla p(\rho)+\rho\nabla V+\tau^{-1}J)dxdt,
\end{align*}
then by using the polar factorization \eqref{eq:assoc},
\[
\RE(\nabla\bar\psi\otimes\nabla\psi)=\nabla\sqrt\rho\otimes\nabla\sqrt\rho+\rho v\otimes v,
\]
we prove the momentum equation. Moreover, by Proposition \ref{prop:3.2} and Proposition \ref{prop:bdI}, we see that $(\rho,v)$ is a GCP solution as in Definition \ref{def:GCPsln} and satisfies the properties in Theorem \ref{thm:glob2}.

Finally, to prove uniqueness, let $(\rho_1, v_1)$ be an arbitrary GCP solution to \eqref{eq:QHD} on $[0,T]$ satisfying the properties of Theorem \ref{thm:glob2}. Then, by the wave function lifting result in Proposition \ref{prop:lift2}, there exists $\psi_1\in X(T)$ associated with $(\rho_1, v_1)$, and $\psi_1$ solves the NLS equation \eqref{eq:cauchy_NLS}. Since $(\rho_1,v_1)$ and $(\rho,v)$ have the same initial data, the corresponding lifted wave functions have the same modulus and velocity at $t=0$. We choose the same constant $S_*$ in the lifting construction of Proposition \ref{prop:lift2}; then the lifted initial wave function associated with $(\rho_1,v_1)$ coincides with $\psi_0$. By the uniqueness result for the Schr\"odinger--Langevin equation \eqref{eq:cauchy_NLS}, we obtain $\psi_1=\psi$. Consequently, the associated hydrodynamic variables coincide, and the uniqueness of GCP solutions to the Cauchy problem for \eqref{eq:QHD} follows.

\end{proof}

\section{Rescaling of the QHD system and the relaxation-time limit}\label{sect:rs}

The focus of this section is the rescaled QHD system \eqref{eq:QHD_rs_intro} and its relaxation-time limit. Recall the scaling
\begin{equation}\label{eq:rs}
t'=\tau t,\quad (\rho_\tau,v_\tau)(t',x)=\left(\rho,\frac{1}{\tau}v\right)\left(\frac{t'}{\tau},x\right).
\end{equation}
Then system \eqref{eq:QHD} can be rewritten as
\begin{equation}\label{eq:QHD_rs}
\left\{\begin{aligned}
&\d_{t'}\rho_\tau+\diver (\rho_\tau v_\tau)=0\\
&\tau^2\d_{t'} (\rho_\tau v_\tau)+\tau^2\diver(\rho_\tau v_\tau\otimes v_\tau)+\nabla p(\rho_\tau)+\rho_\tau\nabla V_\tau=\frac{1}{2}\rho_\tau\nabla\left(\frac{\triangle\sqrt\rho_\tau}{\sqrt\rho_\tau}\right)-\rho_\tau v_\tau\\
&-\triangle V_\tau=\rho_\tau-\mathcal{C}(x),\quad (\rho_\tau,v_\tau)(0,x)=(\rho_0,v_{\tau,0})(x).
\end{aligned}\right.
\end{equation}
and we focus on the case $\mathcal{C}(x)=M_0$. 
For simplicity of notation, in the remainder of this paper we still use $t$ to denote the rescaled time.

After the scaling, the total mass and energy functionals introduced in Section \ref{sect:def} are reformulated in the new coordinates as
\begin{equation}\label{eq:M_rs}
M_\tau(t)=\int_{\T^2} \rho_\tau(t)dx\equiv M_0,
\end{equation}
and
\begin{equation}
E_\tau(t)=\int_{\T^2} e_\tau(t)dx,
\end{equation}
where the rescaled energy density is given by
\begin{equation}\label{eq:endens_rs}
e_\tau=\frac{\tau^2}{2}\rho_\tau |v_\tau|^2+\frac12|\nabla\sqrt{\rho_\tau}|^2+f(\rho_\tau)+\frac12|\nabla V_\tau|^2.
\end{equation}
Also, we have the rescaled chemical potential
\begin{equation}\label{eq:chem_rs}
\mu_\tau=-\frac{\triangle\sqrt\rho_\tau}{2\sqrt\rho_\tau}+\frac{\tau^2}{2}|v_\tau|^2+f'(\rho_\tau)+V_\tau,\quad \sigma_\tau=\d_t\log\sqrt{\rho_\tau}=
-\frac{\diver(\rho_\tau v_\tau)}
{2\rho_\tau}.
\end{equation}
Since the original quantity $\sigma$ and the rescaled quantity $\sigma_\tau$ satisfy
\[
\sigma=\tau\sigma_\tau,
\]
the rescaled GCP quantities are given by
\begin{equation}\label{eq:higher_rs}
I_\tau(t)=\int_{\T^2}\frac12\rho_\tau(\mu_\tau^2+\tau^2\sigma_\tau^2)dx.
\end{equation}
Last, the physical entropy in the rescaled coordinates is defined as
\begin{equation}\label{eq:entr_1}
H(\rho_\tau)=\int_{\T^2}\rho_\tau\log\left(\frac{\rho_\tau}{M_0}\right)dx.
\end{equation}

We now present the estimates satisfied by the rescaled functionals. They are the direct rescaled counterparts of the estimates obtained in Section \ref{sect:exist2d}, and we record them here for completeness.

\begin{prop}\label{prop:rs_est}
\begin{itemize}
\item[(1)] $E_\tau(t)$ satisfies the  rescaled energy balance law
\begin{equation}\label{eq:en_rs}
E_\tau(t)+\int_0^t\int_{\T^2} \rho_\tau |v_\tau|^2dxds=E_0.
\end{equation}

\item[(2)]  The rescaled GCP functional $I_\tau(t)$ satisfies
\begin{equation}\label{eq:ineqI_rs}
\begin{aligned}
\frac{d}{dt} & I_\tau(t)+\frac{1}{4}\int_{\T^2}\rho_\tau\sigma_\tau^2dx+\frac{\tau^2}{2}\int_{\T^2}\rho_\tau |v_\tau|^4 dx \\
\leq & C_0g_3(M_0,E_0,\delta)\int_{\T^2}\rho_\tau\mu_\tau^2dx+C_0\delta^{-1}E_0[M_0+|f'(\delta)|+I_\tau(t)]\int_{\T^2} \rho_\tau |v_\tau|^2dx,
\end{aligned}
\end{equation}
where $g_3(M_0,E_0,\delta)$ is given by \eqref{eq:g_2d}.

\item[(3)] The time derivative of $H(\rho_\tau)$ satisfies
\begin{equation}\label{eq:ineq_entr_rs}
\begin{aligned}
\frac{d}{dt}&\left[H(\rho_\tau)+\tau^2\int_{\T^2}\log\rho_\tau\d_t\rho_\tau\,dx\right]
+\frac{1}{2}\int_{\T^2}\rho_\tau|\nabla^2\log\sqrt\rho_\tau|^2\,dx\\
&+4\int_{\T^2} p'(\rho_\tau)|\nabla\sqrt\rho_\tau|^2\,dx
+\int_{\T^2}(\rho_\tau-M_0)^2\,dx\\
&\leq4\tau^2\int_{\T^2}\rho_\tau\sigma_\tau^2\,dx
+2\tau^4\int_{\T^2} \rho_\tau |v_\tau|^4\,dx.
\end{aligned}
\end{equation}
\end{itemize} 
\end{prop}

Using the rescaled formulas \eqref{eq:en_rs}, \eqref{eq:ineqI_rs}, and \eqref{eq:ineq_entr_rs}, we obtain the corresponding estimates for the combined functional
\[
F_\tau(t)=H(\rho_\tau)+ E_\tau(t)+c_1 I_\tau(t),
\]
where $c_1>0$ is chosen as in Proposition \ref{prop:F_2d}, after rewriting the estimates in the rescaled variables. Finally, under the assumptions of Theorem \ref{thm:glob2}, the functional $I_\tau(t)$ remains uniformly bounded for all $0<\tau\leq\tau^*$.

\subsection{The relaxation-time limit of the rescaled QHD system for GCP solutions}

In this section, we will rigorously prove the relaxation-time limit as $\tau\to 0$ in the framework of GCP solutions with positive density. Moreover, an explicit convergence rate is obtained under the assumption that the functional $I_\tau(t)$ is uniformly bounded, and no additional smallness assumption is required for the relaxation-time limit, beyond the uniform bounds imposed below.

Assume that $\{(\rho_\tau,v_\tau)\}_{\tau\in\mathcal I}$ is a sequence of GCP solutions to the rescaled QHD system \eqref{eq:QHD_rs} satisfying the following conditions on $[0,T)\times\T^2$:
\begin{itemize}
\item[(1)] $\inf_{t,x}\rho_\tau\geq \delta>0$;
\item[(2)] for all $t\in[0,T)$, $(\rho_\tau,v_\tau)$ satisfies the uniform mass-energy bounds
\begin{equation}\label{eq:bd_me_rxl}
M_\tau(t)=M_0,\quad E_\tau(t)+\int_0^t\int_{\T^2} \rho_\tau |v_\tau|^2dxds\leq E_0,
\end{equation}
and the uniform higher-order energy bounds:
\begin{equation}\label{eq:bd_I_rxl}
I_\tau(t)+\int_0^t\int_{\T^2} \rho_\tau\sigma_\tau^2 dxds+\tau^2\int_0^t\int_{\T^2} \rho_\tau |v_\tau|^4 dxds\leq C(\delta,M_0,E_0,I_0),
\end{equation}
where $M_\tau(t)$, $E_\tau(t)$ and $I_\tau(t)$ are rescaled functionals given in the first part of Section \ref{sect:rs}.
\end{itemize}

We first prove the following Proposition concerning the weak relaxation-time limit.

\begin{prop}
Let $\{(\rho_\tau,v_\tau)\}_{\tau\in\mathcal{I}}$ be a sequence of GCP solutions satisfying conditions $(1)$ and $(2)$, and set $J_\tau=\rho_\tau v_\tau$. Then there exists a subsequence, denoted by $\{(\rho_{\tau_n},J_{\tau_n})\}$, and a limiting density $\bar\rho$ and momentum density $\bar J$ such that the relaxation-time limit holds in the sense
\begin{equation}\label{eq:weaklim}
\sqrt\rho_{\tau_n}\rightharpoonup\sqrt{\bar\rho}\quad\textrm{in }L^2_{t,loc}H^2_x,\quad J_{\tau_n}\rightharpoonup \bar J\quad\textrm{in }L^2_{t,loc}L^2_x,
\end{equation}
where $\bar\rho$ is a weak solution of the QDD equation \eqref{eq:qdde} in the sense of Definition \ref{def:qdde_ws}, and $\bar J$ satisfies the consistency relation
\begin{equation}\label{eq:cons_J}
\bar J=\frac14\diver(\bar\rho\nabla^2\log\bar\rho)-\nabla p(\bar\rho)-\bar\rho \nabla\bar V.
\end{equation}
\end{prop}

\begin{proof}
The assumptions above imply that $(\rho_\tau,v_\tau)$ satisfies the uniform bounds
\begin{equation}\label{eq:prop29.1}
\|\sqrt\rho_\tau\|_{L^\infty_tH^1_x}\leq C(M_0,E_0),\quad \|\sqrt\rho_\tau v_\tau\|_{L^2_{t,x}}\leq E_0^\frac12
\end{equation}
on $[0,T)\times \T^2$.
The endpoint contribution of the correction term is bounded by $C\tau^2$ thanks to the uniform lower and upper bounds on $\rho_\tau$ and the estimate \eqref{eq:bd_I_rxl}. Therefore, by \eqref{eq:ineq_entr_rs} and Lemma \ref{lem:H^2log} we have
\begin{equation}\label{eq:7-1}
\begin{aligned}
H(\rho_\tau)(t)&+\int_0^t\int_{\T^2}|\nabla^2\sqrt\rho_\tau|^2+|\nabla\rho_\tau^\frac14|^4dxds\\
&+\int_0^t\int_{\T^2} p'(\rho_\tau)|\nabla\sqrt\rho_\tau|^2dxds\leq H(\rho_0)+C(M_0,E_0,I_0,\delta)\tau^2.
\end{aligned}
\end{equation}
Since our internal energy $f(\rho)=(2n)^{-1}(\rho-M_0)^{2n}$ is convex, we have
\[
p'(\rho)=\rho f''(\rho)\geq 0.
\]
For the initial entropy $H(\rho_0)$, we have
\[
H(\rho_0)=\int_{\T^2} \rho_0\log\left(\frac{\rho_0}{M_0}\right)dx\leq C(M_0,E_0,\delta).
\]
Hence, on $[0,T)\times\T^2$, we obtain
\[
\|\nabla^2\sqrt\rho_\tau\|_{L^2_{t,x}}\leq C(M_0,E_0,\delta)+C(M_0,E_0,I_0,\delta)\tau.
\]
Therefore we can choose a converging subsequence $\{(\rho_{\tau_n},v_{\tau_n})\}$, $\tau_n\to 0$, and a limiting function $\xi\geq 0$, such that
\[
\sqrt\rho_{\tau_n}\rightharpoonup\xi\quad\textrm{in }L^2_{t,loc}H^2_x.
\]
Moreover, since $\d_t\sqrt{\rho_\tau}=\sqrt{\rho_\tau}\sigma_\tau$ is uniformly bounded in $L^2_{t,x}$, the Aubin--Lions compactness lemma yields the strong convergence
\[
\sqrt\rho_{\tau_n}\to\xi\quad\textrm{in }L^2_{t,loc}H^1_x.
\]
This implies that $\rho_{\tau_n}\to \xi^2$, and we can therefore write $\xi=\sqrt{\bar\rho}$. Now let $\eta_1\in\mathcal C^\infty_0([0, T)\times\T^2)$ and $\eta_2\in\mathcal C^\infty_0([0, T)\times\T^2;\R^2)$ be test functions. The weak formulation of the rescaled QHD system \eqref{eq:QHD_rs} reads
\[
\int_0^T\int_{\T^2}\rho_{\tau_n}\d_t\eta_1+J_{\tau_n}\cdot\nabla\eta_1\,dxdt+\int_{\T^2}\rho_0(x)\eta_1(0, x)\,dx=0;
\]
and
\begin{align*}
\int_0^T\int_{\T^2}\tau_n^2J_{\tau_n}\cdot\d_t\eta_2&+(\tau_n^2\Lambda_{\tau_n}\otimes\Lambda_{\tau_n}+p(\rho_{\tau_n})I_d+\nabla\sqrt\rho_{\tau_n}\otimes\nabla\sqrt\rho_{\tau_n}):\nabla\eta_2
\\
&-\eta_2\cdot\rho_{\tau_n}\nabla V_{\tau_n}-\frac{1}{4}\rho_{\tau_n}\triangle\diver\eta_2-J_{\tau_n}\cdot\eta_2\,dxdt+\tau_n\int_{\T^2}J_{0}(x)\cdot\eta_2(0, x)\,dx=0,
\end{align*}
where $J_{\tau_n}=\sqrt\rho_{\tau_n}\Lambda_{\tau_n}=\rho_{\tau_n} v_{\tau_n}$. 
The bound \eqref{eq:prop29.1} on $\sqrt\rho_{\tau_n}v_{\tau_n}$ and the strong convergence of $\sqrt\rho_{\tau_n}$ imply the convergence 
\[
J_{\tau_n}=\rho_{\tau_n} v_{\tau_n}\rightharpoonup \bar J\quad \textrm{in }L^2_{t,loc}L^2_x,\quad V_{\tau_n}\to\bar V\quad \textrm{in }L^2_{t,loc}H^1_x.
\]
Thus we can pass to the limit ${\tau_n}\to 0$ to obtain
\begin{equation}\label{eq:weak_QDD_1}
\int_0^T\int_{\T^2}\bar\rho\d_t\eta_1+\bar J\cdot\nabla\eta_1 dxdt+\int_{\T^2} \rho_0\eta_1(0)dx=0
\end{equation}
and
\begin{equation}\label{eq:weak_QDD2}
\int_0^T\int_{\T^2}\bar J\cdot\eta_2dxdt=\int_0^T\int_{\T^2} [\nabla\sqrt{\bar\rho}\otimes\nabla\sqrt{\bar\rho}+p(\bar\rho)I_d]:\nabla\eta_2-\bar\rho\nabla\bar V\cdot\eta_2-\frac14\bar\rho\triangle\diver\eta_2 dxdt.
\end{equation}
Since $\eta_2$ is an arbitrary test vector function, \eqref{eq:weak_QDD2} implies \eqref{eq:cons_J}. Last, by choosing $\eta_2=\nabla\eta_1$ in \eqref{eq:weak_QDD2} and substituting it into \eqref{eq:weak_QDD_1}, we obtain the weak formulation of equation \eqref{eq:qdde},
\begin{align*}
\int_0^T\int_{\T^2}&\bar\rho\d_t\eta_1+[\nabla\sqrt{\bar\rho}\otimes\nabla\sqrt{\bar\rho}+p(\bar\rho)I_d]:\nabla^2\eta_1\\
&-\bar\rho\nabla\bar V\cdot\nabla\eta_1-\frac14\bar\rho\triangle^2\eta_1 dxdt+\int_{\T^2} \rho_0\eta_1(0)dx=0,
\end{align*}
as in Definition \ref{def:qdde_ws}.
\end{proof}

As a consequence of the weak relaxation-time limit, the solution $\bar\rho$ of the QDD equation \eqref{eq:qdde} obtained above satisfies the following properties.

\begin{prop}\label{prop:barrho}
The limiting density $\bar\rho$ has the lower bound $\inf\bar\rho\geq\delta>0$, and it satisfies the following bounds:
\begin{equation}\label{eq:bd_barrho_1}
\|\sqrt{\bar\rho}\|_{L^\infty_tH^1_x}\leq C(M_0,E_0),
\end{equation}
and 
\begin{equation}\label{eq:bd_barrho_2}
\|\sqrt{\bar\rho}\|_{L^2_{t}H^4_x}+\|\nabla\bar\rho^\frac14\|_{L^4_{t,x}}\leq C(\delta,M_0,E_0,I_0).
\end{equation}
\end{prop}

\begin{proof}
The lower bound for $\bar\rho$ follows from the strong convergence of $\rho_{\tau_n}$ and the uniform lower bound $\rho_{\tau_n}\geq\delta$. By passing to the limit in the energy balance law \eqref{eq:bd_me_rxl}, we obtain \eqref{eq:bd_barrho_1}.

For \eqref{eq:bd_barrho_2}, we first notice that the $L^2_{t,x}$ control of $\nabla^2\sqrt{\bar\rho}$, together with the bound of $\nabla\bar\rho^\frac14$, follows from the rescaled entropy estimate \eqref{eq:7-1}. To obtain the bound of the higher derivative of $\sqrt{\bar\rho}$, we first recall that the consistent momentum density $\bar J$ given by \eqref{eq:cons_J} inherits the $L^2_{t,x}$ bound of $J_\tau$ by the weak limit \eqref{eq:weaklim}, which together with \eqref{eq:bohm} implies
\begin{align*}
\|\nabla\triangle\bar\rho\|_{L^2_{t,x}}\leq & C(\|\bar J\|_{L^2_{t,x}}+\|\nabla p(\bar\rho)\|_{L^2_{t,x}}+\|\bar\rho\nabla \bar V\|_{L^2_{t,x}})\\
\leq & C(\|\bar J\|_{L^2_{t,x}},\|\bar\rho\|_{L^2_{t}H^2_x})\leq C(\delta, M_0,E_0,I_0).
\end{align*}
Then, by standard elliptic estimate and the lower bound of $\bar\rho$, we obtain
\[
\|\sqrt{\bar\rho}\|_{L^2_{t}H^3_x}\leq C(\delta, M_0,E_0,I_0).
\]

Moreover, by the uniform bound of $\d_t\sqrt{\rho_\tau}=\sqrt\rho_\tau\sigma_\tau$, it follows that
\begin{equation}\label{eq:bd_barrho_3}
\|\d_t\sqrt{\bar\rho}\|_{L^2_{t,x}}\leq C(\delta,M_0,E_0,I_0).
\end{equation}
By \eqref{eq:bohm}, we can write the quantum drift-diffusion equation as
\[
\d_t\bar\rho+\frac14\triangle^2\bar\rho+\diver\cdot\diver(\nabla\sqrt{\bar\rho}\otimes\nabla\sqrt{\bar\rho})-\triangle p(\bar\rho)-\diver(\bar\rho\nabla\bar V)=0.
\]
Thus we have
\[
\|\triangle^2\bar\rho\|_{L^2_{t,x}}
\leq C\bigl(\|\d_t\bar\rho\|_{L^2_{t,x}},\|\sqrt{\bar\rho}\|_{L^2_tH^3_x}\bigr)
\leq C(\delta,M_0,E_0,I_0).
\]
By elliptic estimates, this gives $\bar\rho\in L^2_tH^4_x$. Since $\bar\rho\geq\delta>0$, the composition $\sqrt{\bar\rho}$ also belongs to $L^2_tH^4_x$, and therefore \eqref{eq:bd_barrho_2} follows.
\end{proof}

We now prove a relative entropy estimate which, in particular, yields convergence of the whole family $\{\rho_\tau\}$ to the limiting QDD solution $\bar\rho$, and an explicit convergence rate is also obtained. As the main technical tool of this part, we use the method of relative entropy. Recall that the physical entropy $H(\rho)(t)$ is given by
\[
H(\rho)(t)=\int_{\T^2} g(\rho(t))dx,
\]
where $g(s)=s \log(s/M_0)$. For $\rho_\tau$ and the limiting density $\bar\rho$, we define the relative entropy as
\begin{equation}\label{eq:rl_entr}
\begin{aligned}
H(\rho_\tau|\bar\rho)(t)=&H(\rho_\tau)(t)-H(\bar\rho)(t)-\int_{\T^2}  g'(\bar\rho)(\rho_\tau-\bar\rho)dx\\
=&\int_{\T^2}  g(\rho_\tau)-g(\bar\rho)-g'(\bar\rho)(\rho_\tau-\bar\rho)dx,
\end{aligned}
\end{equation} 
where
\[
g'(s)=\log\left(\frac{s}{M_0}\right)+1.
\]
By Taylor's formula, we have
\[
g(\rho_\tau)-g(\bar\rho)-g'(\bar\rho)(\rho_\tau-\bar\rho)=\frac12 g''(\rho_*)(\rho_\tau-\bar\rho)^2
\]
for some $\rho_*\in [\min\{\rho_\tau,\bar\rho\},\max\{\rho_\tau,\bar\rho\}]$, and $g''(s)=s^{-1}$. Moreover, both $\rho_\tau$ and $\bar\rho$ are bounded from below by $\delta>0$ and from above by a constant depending on $(M_0,E_0,I_0,\delta)$. Thus the relative entropy $H(\rho_\tau|\bar\rho)$ is equivalent to the $L^2_x$ norm of $\rho_\tau-\bar\rho$ in the sense
\begin{equation}\label{eq:equiv_entr}
2 \delta H(\rho_\tau|\bar\rho)(t)\leq\|\rho_\tau-\bar\rho\|_{L^2_x}^2(t)\leq 2 C(M_0,E_0,I_0,\delta)H(\rho_\tau|\bar\rho)(t).
\end{equation}

To estimate the relative entropy $H(\rho_\tau|\bar\rho)$, we need the following result, which gives the time derivative of $H(\rho_\tau|\bar\rho)$. Here again, a rigorous computation requires a mollification argument of $\rho_\tau$ as in the proof of Proposition \ref{prop:dtentr}, while the regularity of the limiting density $\bar\rho$ is guaranteed by Proposition \ref{prop:barrho}. Since the method is similar, we omit this step for simplicity.

\begin{prop}\label{prop:6.2}
Let $\{(\rho_\tau,v_\tau)\}$ be a sequence of solutions to \eqref{eq:QHD_rs}, and let $\bar\rho$ be the limiting density obtained above, satisfying the bounds in Proposition \ref{prop:barrho}. For almost every $t\in(0,T)$, the time derivative of $H(\rho_\tau|\bar\rho)$ is given by
\begin{equation}\label{eq:dtrlentr}
\begin{aligned}
\frac{d}{dt}H(\rho_\tau|\bar\rho)=&-\int_{\T^2} \rho_\tau\left|\nabla^2\log\sqrt\rho_\tau-\nabla^2\log\sqrt{\bar\rho}\right|^2dx\\
&+2\int_{\T^2} \rho_\tau\nabla^2\log\sqrt{\bar\rho}:(\nabla\log\sqrt\rho_\tau-\nabla\log\sqrt{\bar\rho})^2dx\\
&-\int_{\T^2} \rho_\tau[p'(\rho_\tau)\nabla\log\rho_\tau-p'(\bar\rho)\nabla\log\bar\rho]\cdot(\nabla\log\rho_\tau-\nabla\log\bar\rho) dx\\
&-\int_{\T^2}(\rho_\tau-\bar\rho)^2dx+\int_{\T^2}(\rho_\tau-\bar\rho)\nabla\log\bar\rho\cdot\nabla(V_\tau-\bar V)dx\\
&+\tau^2[R_1(\rho_\tau)-R_2(\rho_\tau,\bar\rho)],
\end{aligned}
\end{equation}
where we use the shorthand notation $\mathbf w^2=\mathbf w\otimes\mathbf w$ for a vector $\mathbf w$, and the remaining terms are given by
\[
R_1(\rho_\tau)=-\frac{d}{dt}\int_{\T^2}\log\rho_\tau\d_t\rho_\tau dx+4\int_{\T^2}\rho_\tau\sigma_\tau^2 dx+\int_{\T^2}\rho_\tau (v_\tau\otimes v_\tau):\nabla^2\log\rho_\tau dx,
\]
and
\[
R_2(\rho_\tau,\bar\rho)=-\frac{d}{dt}\int_{\T^2}\log\bar\rho\d_t\rho_\tau dx+\int_{\T^2}\d_t\log\bar\rho\d_t\rho_\tau dx+\int_{\T^2}\rho_\tau (v_\tau\otimes v_\tau):\nabla^2\log\bar\rho dx.
\]
\end{prop}

\begin{proof}
The time derivative of $H(\rho_\tau|\bar\rho)$ is computed by
\[
\frac{d}{dt}H(\rho_\tau|\bar\rho)=\frac{d}{dt}H(\rho_\tau)-\int_{\T^2} \d_t\rho_\tau\log\bar\rho dx-\int_{\T^2} \rho_\tau\d_t\log\bar\rho dx,
\]
where the terms containing $\int_{\T^2}\d_t\rho_\tau$ and $\int_{\T^2}\d_t\bar\rho$ vanish by conservation of mass. The time derivative $\frac{d}{dt}H(\rho_\tau)$ can be obtained by applying the scaling \eqref{eq:rs} to Proposition \ref{prop:dtentr}, which is given by
\begin{align*}
\frac{d}{dt}H(\rho_\tau)=&-\int_{\T^2}\rho_\tau|\nabla^2\log\sqrt\rho_\tau|^2 dx-\int_{\T^2} p'(\rho_\tau)\rho_\tau|\nabla\log\rho_\tau|^2dx\\
&-\int_{\T^2}(\rho_\tau-M_0)^2dx+\tau^2 R_1(\rho_\tau).
\end{align*}
For the remaining terms, by using \eqref{eq:QHD_rs} and \eqref{eq:qdde}, we have
\begin{align*}
-\int_{\T^2}\d_t\rho_\tau\log\bar\rho dx=&\frac12\int_{\T^2}\log\bar\rho\diver\cdot\diver(\rho_\tau\nabla^2\log\sqrt\rho_\tau)dx-\int_{\T^2}\log\bar\rho\triangle p(\rho_\tau)dx\\
&-\int_{\T^2}\log\bar\rho\diver(\rho_\tau\nabla V_\tau) dx+\tau^2\int_{\T^2}\log\bar\rho\d_t^2\rho_\tau dx\\
&-\tau^2\int_{\T^2}\log\bar\rho\diver\cdot\diver(\rho_\tau v_\tau\otimes v_\tau)dx\\
=&\int_{\T^2}\rho_\tau\nabla^2\log\sqrt{\bar\rho}:\nabla^2\log\sqrt\rho_\tau dx+\int_{\T^2} p'(\rho_\tau)\rho_\tau\nabla\log\rho_\tau\cdot\nabla\log\bar\rho dx\\
&+\int_{\T^2}\rho_\tau\nabla\log\bar\rho\cdot\nabla V_\tau dx-\tau^2 R_2(\rho_\tau,\bar\rho),
\end{align*}
where 
\begin{align*}
R_2(\rho_\tau,\bar\rho)=&\int_{\T^2}\rho_\tau (v_\tau\otimes v_\tau):\nabla^2\log\bar\rho dx-\int_{\T^2}\log\bar\rho\d_t^2\rho_\tau dx\\
=&\int_{\T^2}\rho_\tau (v_\tau\otimes v_\tau):\nabla^2\log\bar\rho dx-\frac{d}{dt}\int_{\T^2}\log\bar\rho\d_t\rho_\tau dx+\int_{\T^2}\d_t\log\bar\rho\d_t\rho_\tau dx,
\end{align*}
and
\begin{align*}
-\int_{\T^2} \rho_\tau\d_t\log\bar\rho dx=&\frac12\int_{\T^2} \frac{\rho_\tau}{\bar\rho}\diver\cdot\diver(\bar\rho\nabla^2\log\sqrt{\bar\rho})dx-\int_{\T^2} \frac{\rho_\tau}{\bar\rho}\triangle p(\bar\rho)dx\\
&-\int_{\T^2} \frac{\rho_\tau}{\bar\rho}\diver(\bar\rho\nabla\bar V)dx\\
=&\frac12\int_{\T^2} \bar\rho\nabla^2\log\sqrt{\bar\rho}:\nabla^2\left(\frac{\rho_\tau}{\bar\rho}\right)dx+\int_{\T^2} \nabla p(\bar\rho)\cdot\nabla\left(\frac{\rho_\tau}{\bar\rho}\right)dx\\
&+\int_{\T^2} \rho_\tau(\bar\rho-\mathcal{C}(x))dx-\int_{\T^2} \rho_\tau\nabla\log\bar\rho\cdot\nabla\bar Vdx.
\end{align*}
By substituting 
\[
\nabla\left(\frac{\rho_\tau}{\bar\rho}\right)=\frac{\rho_\tau}{\bar\rho}(\nabla\log\rho_\tau-\nabla\log\bar\rho)
\]
and
\begin{align*}
\nabla^2\left(\frac{\rho_\tau}{\bar\rho}\right)=&\frac{\rho_\tau}{\bar\rho}(\nabla^2\log\rho_\tau-\nabla^2\log\bar\rho)+\frac{\rho_\tau}{\bar\rho}(\nabla\log\rho_\tau-\nabla\log\bar\rho)^2
\end{align*}
into the previous identity, we obtain
\begin{align*}
-\int_{\T^2} \rho_\tau\d_t\log\bar\rho dx=&\int_{\T^2} \rho_\tau\nabla^2\log\sqrt{\bar\rho}:(\nabla^2\log\sqrt\rho_\tau-\nabla^2\log\sqrt{\bar\rho})dx\\
&+2\int_{\T^2} \rho_\tau\nabla^2\log\sqrt{\bar\rho}:(\nabla\log\sqrt\rho_\tau-\nabla\log\sqrt{\bar\rho})^2dx\\
&+\int_{\T^2} p'(\bar\rho)\rho_\tau\nabla\log\bar\rho\cdot(\nabla\log\rho_\tau-\nabla\log\bar\rho)dx.
\end{align*}
To obtain \eqref{eq:dtrlentr}, we summarize the computations above, and notice that the integrals containing $V_\tau$ and $\bar V$ are
\begin{align*}
&-\int_{\T^2}\rho_\tau(\rho_\tau-\bar\rho)dx+\int_{\T^2}\rho_\tau\nabla\log\bar\rho\cdot\nabla(V_\tau-\bar V)dx\\
=&-\int_{\T^2}\rho_\tau(\rho_\tau-\bar\rho)dx+\int_{\T^2}(\rho_\tau-\bar\rho)\nabla\log\bar\rho\cdot\nabla(V_\tau-\bar V)dx+\int_{\T^2}\nabla\bar\rho\cdot\nabla(V_\tau-\bar V)dx\\
=&-\int_{\T^2}\rho_\tau(\rho_\tau-\bar\rho)dx+\int_{\T^2}(\rho_\tau-\bar\rho)\nabla\log\bar\rho\cdot\nabla(V_\tau-\bar V)dx-\int_{\T^2}\bar\rho\triangle(V_\tau-\bar V)dx\\
=&-\int_{\T^2}(\rho_\tau-\bar\rho)^2dx+\int_{\T^2}(\rho_\tau-\bar\rho)\nabla\log\bar\rho\cdot\nabla(V_\tau-\bar V)dx.
\end{align*}
\end{proof}

\begin{proof}[Proof of Theorem \ref{thm:rlxlimit}]
It remains to control the difference terms on the right-hand side of \eqref{eq:dtrlentr} by the relative entropy $H(\rho_\tau|\bar\rho)$ and the dissipation. For simplicity of notation, we denote by $C$ a generic positive constant independent of $\tau$. Throughout this proof, $C$ is chosen sufficiently large depending only on the right-hand side of \eqref{eq:bd_I_rxl}.

We first write
\begin{align*}
\int_0^t\int_{\T^2} \rho_\tau\nabla^2\log\sqrt{\bar\rho}&:(\nabla\log\sqrt\rho_\tau-\nabla\log\sqrt{\bar\rho})^2dxds\\
\leq & \delta^{-\frac12}\|\rho_\tau\|_{L^\infty_tL^4_x}\|\sqrt{\bar\rho}\nabla^2\log\sqrt{\bar\rho}\|_{L^2_tL^4_x}\|\nabla\log\sqrt\rho_\tau-\nabla\log\sqrt{\bar\rho}\|^2_{L^2_tL^4_x}.
\end{align*}
It follows from the interpolation inequality that
\begin{equation}\label{eq:bd_log1}
\begin{aligned}
\|\rho_\tau\|_{L^\infty W^{1,4}_x}\leq & C_0\|\sqrt\rho_\tau\|_{L^\infty H^2_x}^2\leq C,\\
\|\sqrt{\bar\rho}\nabla^2\log\sqrt{\bar\rho}\|_{L^2_tL^4_x}\leq &C_0\|\nabla\sqrt{\bar\rho}\|_{L^2H^2_x}\leq C,\\
\|\nabla\log\sqrt\rho_\tau-\nabla\log\sqrt{\bar\rho}\|^2_{L^2_tL^4_x}\leq & C_0\|\log\sqrt\rho_\tau-\log\sqrt{\bar\rho}\|_{L^2_{t,x}}^2\\
&+C_0^{-1}\|\nabla^2\log\sqrt\rho_\tau-\nabla^2\log\sqrt{\bar\rho}\|_{L^2_{t,x}}^2.
\end{aligned}
\end{equation}
Moreover, by using
\begin{equation}\label{eq:bd_log2}
|\log\sqrt\rho_\tau-\log\sqrt{\bar\rho}|\leq \frac{|\rho_\tau-\bar\rho|}{2\min\{\rho_\tau,\bar\rho\}}\leq \frac{|\rho_\tau-\bar\rho|}{2\delta},
\end{equation}
and the equivalence \eqref{eq:equiv_entr} of $H(\rho_\tau|\bar\rho)$ and $\|\rho_\tau-\bar\rho\|_{L^2_x}^2$, we obtain
\begin{align*}
\int_0^t\int_{\T^2} \rho_\tau\nabla^2\log\sqrt{\bar\rho}&:(\nabla\log\sqrt\rho_\tau-\nabla\log\sqrt{\bar\rho})^2dxds\\
\leq & C\int_0^t H(\rho_\tau|\bar\rho)ds+\frac14\int_0^t\int_{\T^2}\rho_\tau|\nabla^2\log\sqrt\rho_\tau-\nabla^2\log\sqrt{\bar\rho}|^2 dxds.
\end{align*}

To estimate the pressure term in the right-hand side of \eqref{eq:dtrlentr}, we write
\begin{align*}
p'(\rho_\tau)\nabla\log\rho_\tau-p'(\bar\rho)\nabla\log\bar\rho=&[p'(\rho_\tau)-p'(\bar\rho)]\nabla\log\rho_\tau\\
&+p'(\bar\rho)(\nabla\log\rho_\tau-\nabla\log\bar\rho),
\end{align*}
which gives
\begin{align*}
-\int_0^t\int_{\T^2} \rho_\tau&[p'(\rho_\tau)\nabla\log\rho_\tau-p'(\bar\rho)\nabla\log\bar\rho]\cdot(\nabla\log\rho_\tau-\nabla\log\bar\rho) dxds\\
\leq & \int_0^t\int_{\T^2} |[p'(\rho_\tau)-p'(\bar\rho)]\nabla\rho_\tau\cdot(\nabla\log\rho_\tau-\nabla\log\bar\rho) |dxds\\
&+\int_0^t\int_{\T^2}\rho_\tau|p'(\bar\rho)||\nabla\log\rho_\tau-\nabla\log\bar\rho|^2dxds=\textrm{I}+\textrm{II}.
\end{align*}
Since $p\in C^2(0,+\infty)$, we can write
\[
p'(\rho_\tau)-p'(\bar\rho)=p''(\rho_*)(\rho_\tau-\bar\rho),
\]
for some pointwise value $\rho_*$ between $\rho_\tau$ and $\bar\rho$. 
We estimate the integral $\textrm{I}$ as
\begin{align*}
\textrm{I}\leq & \|\nabla\rho_\tau\|_{L^\infty_tL^4_x}\|p''(\rho_*)\|_{L^\infty_{t,x}}\|\rho_\tau-\bar\rho\|_{L^2_{t,x}}\|\nabla\log\rho_\tau-\nabla\log\bar\rho\|_{L^2_tL^4_x}\\
\leq & C\|\rho_\tau-\bar\rho\|_{L^2_{t,x}}^2+\frac14\|\sqrt\rho_\tau(\nabla^2\log\sqrt\rho_\tau-\nabla^2\log\sqrt{\bar\rho})\|_{L^2_{t,x}}^2\\
\leq & C\int_0^tH(\rho_\tau|\bar\rho)ds+\frac14\int_0^t\int_{\T^2}\rho_\tau\left|\nabla^2\log\sqrt\rho_\tau-\nabla^2\log\sqrt{\bar\rho}\right|^2dxds,
\end{align*}
where we again use \eqref{eq:bd_log1} and \eqref{eq:bd_log2}. 
Finally, for the term $\textrm{II}$, using $p\in C^2(0,+\infty)$, \eqref{eq:bd_log1}, and \eqref{eq:bd_log2}, we argue similarly to obtain
\begin{align*}
\textrm{II}\leq & C\int_0^t H(\rho_\tau|\bar\rho)ds+\frac14\int_0^t\int_{\T^2}\rho_\tau|\nabla^2\log\sqrt\rho_\tau-\nabla^2\log\sqrt{\bar\rho}|^2 dxds.
\end{align*}
For the integral of electric potentials, we control it by
\begin{align*}
\int_0^t\int_{\T^2}(\rho_\tau-\bar\rho)\nabla\log\bar\rho\cdot\nabla(V_\tau-\bar V)dxds\leq & \|\nabla\log\bar\rho\|_{L^\infty_tL^4_x}\|\rho_\tau-\bar\rho\|_{L^2_{t,x}}\|\nabla(V_\tau-\bar V)\|_{L^2_tL^4_x}\\
\leq& C\|\rho_\tau-\bar\rho\|_{L^2_{t,x}}^2\leq C\int_0^t H(\rho_\tau|\bar\rho)ds,
\end{align*}
where we use the Poisson equation for $V_\tau$, $\bar V$ and Poincar\'e inequality
\[
\|\nabla(V_\tau-\bar V)\|_{L^4_x}\leq C\|\triangle(V_\tau-\bar V)\|_{L^2_x}=C\|\rho_\tau-\bar\rho\|_{L^2_x}.
\]

Summarizing the inequalities above and integrating in time, since $(\rho_\tau,\bar\rho)$ share the same initial data, we have
\begin{align*}
H(\rho_\tau|\bar\rho)(t)+\frac14\int_0^t\int_{\T^2}\rho_\tau &|\nabla^2\log\sqrt\rho_\tau-\nabla^2\log\sqrt{\bar\rho}|^2dxds\\
\leq & C\int_0^t H(\rho_\tau|\bar\rho)ds+\tau^2\int_0^t[R_1(\rho_\tau)-R_2(\rho_\tau,\bar\rho)]ds.
\end{align*}

Now we estimate the last term $\tau^2\int_0^t[R_1(\rho_\tau)-R_2(\rho_\tau,\bar\rho)]ds$. Recall from Proposition \ref{prop:6.2} that 
\begin{equation}\label{eq:6.13}
\begin{aligned}
\tau^2\int_0^t[R_1(\rho_\tau)-R_2(\rho_\tau,\bar\rho)]ds=&-\tau^2\left.\int_{\T^2}(\log\rho_\tau-\log\bar\rho)\d_t\rho_\tau dx\right|_{s=0}^{s=t}\\
&+\tau^2\int_0^t\int_{\T^2}(\nabla^2\log\rho_\tau-\nabla^2\log\bar\rho):(\rho_\tau v_\tau\otimes v_\tau )dxds\\
&+\tau^2\int_0^t\int_{\T^2}(4\rho_\tau\sigma^2_\tau+\d_t\log\bar\rho\d_t\rho_\tau) dxds.
\end{aligned}
\end{equation}
By the Cauchy--Schwarz inequality, the first term is controlled by
\begin{align*}
\tau^2\left|\int_{\T^2}(\log\rho_\tau-\log\bar\rho)\d_t\rho_\tau dx\right|\leq & \kappa^{-1}\tau^4\|\sqrt\rho_\tau\sigma_\tau\|_{L^2_x}^2+\kappa \|\sqrt\rho_\tau(\log\rho_\tau-\log\bar\rho)\|_{L^2_x}^2.
\end{align*}
It follows from \eqref{eq:bd_I_rxl} that
\[
\tau^4\|\sqrt\rho_\tau\sigma_\tau\|_{L^2_x}^2\leq C\,\tau^2,
\]
and by the upper and lower bounds of $\rho_\tau$ and $\bar\rho$, we have
\[
\|\sqrt\rho_\tau(\log\rho_\tau-\log\bar\rho)\|_{L^2_x}^2\sim\|\rho_\tau-\bar\rho\|_{L^2_x}^2\sim H(\rho_\tau|\bar\rho).
\]
Therefore, we can choose $\kappa$ small such that
\[
\tau^2\left|\int_{\T^2}(\log\rho_\tau-\log\bar\rho)\d_t\rho_\tau dx\right|\leq C\,\tau^2+\frac12 H(\rho_\tau|\bar\rho).
\]
For the second term in the right-hand side of \eqref{eq:6.13}, we have
\begin{align*}
\tau^2\int_0^t\int_{\T^2}&(\nabla^2\log\rho_\tau-\nabla^2\log\bar\rho):\rho_\tau v_\tau\otimes v_\tau dxds\\
\leq&\frac18\int_0^t\int_{\T^2}\rho_\tau |\nabla^2\log\sqrt\rho_\tau-\nabla^2\log\sqrt{\bar\rho}|^2dxds+8\tau^4\int_0^t\int_{\T^2}\rho_\tau |v_\tau|^4dxds,
\end{align*}
and by \eqref{eq:bd_I_rxl}, the last term is bounded by 
\[
\tau^4\int_0^t\int_{\T^2}\rho_\tau |v_\tau|^4dxds\leq C\,\tau^2.
\]
The last integral on the right-hand side of \eqref{eq:6.13} is controlled by \eqref{eq:bd_I_rxl} and \eqref{eq:bd_barrho_3} as follows:
\begin{align*}
\tau^2\int_0^t\int_{\T^2}&(4\rho_\tau\sigma^2_\tau+\d_t\log\bar\rho\d_t\rho_\tau) dxds\\
\leq & C\,\tau^2\int_0^t(\|\sqrt\rho_\tau\sigma_\tau\|_{L^2_x}^2+\|\sqrt\rho_\tau\sigma_\tau\|_{L^2_x}\|\d_t\sqrt{\bar\rho}\|_{L^2_x})ds\leq C\,\tau^2.
\end{align*}

Summarizing the arguments above, we obtain
\begin{align*}
\frac12H(\rho_\tau|\bar\rho)(t)+\frac18\int_0^t\int_{\T^2}\rho_\tau &\left|\nabla^2\log\sqrt\rho_\tau-\nabla^2\log\sqrt{\bar\rho}\right|^2dxds\\
\leq & C\int_0^t H(\rho_\tau|\bar\rho)ds+C\,\tau^2,
\end{align*}
where we have absorbed all small fractions of the dissipation into the left-hand side. 
Then, by the integral form of Gronwall's inequality and $H(\rho_\tau|\bar\rho)(0)=0$, we obtain
\[
H(\rho_\tau|\bar\rho)(t)+\int_0^t\int_{\T^2}\rho_\tau \left|\nabla^2\log\sqrt\rho_\tau-\nabla^2\log\sqrt{\bar\rho}\right|^2dxds\leq C\,\tau^2
\]
for all $t\in[0,T)$. By the equivalence \eqref{eq:equiv_entr} between $H(\rho_\tau|\bar\rho)$ and $\|\rho_\tau-\bar\rho\|_{L^2_x}^2$, this yields
\[
\|\rho_\tau-\bar\rho\|_{L^\infty(0,T;L^2_x)}\leq C\tau,
\]
and the proof is complete.
\end{proof}

\section*{Appendix: Hydrodynamic derivation of energy and higher-order balance law}

The energy balance law \eqref{eq:en_disp} and the time derivative \eqref{eq:rl_I} of the higher-order functional $I(t)$ can be derived by a purely hydrodynamic argument based on the dynamics of the QHD system \eqref{eq:QHD_md}. We present these computations in this appendix for completeness.

\begin{prop}\label{lem:dte_apdx}
Let $(\rho, v)$ be a GCP solution to \eqref{eq:QHD_md} with finite energy, such that $\rho>0$. Then the energy density 
\[
e(t,x)=\frac12|\nabla\sqrt\rho|^2+\frac12\rho |v|^2+f(\rho)+\frac12|\nabla V|^2
\]
satisfies the following distributional equation
\begin{equation}\label{eq:en_cons_apdx}
\d_te+\diver(\rho v\mu-\d_t\sqrt{\rho}\nabla\sqrt{\rho}-V\nabla\d_tV)+\frac{1}{\tau}\rho |v|^2=0.
\end{equation}
As a consequence, the total energy $E(t)=\int_{\T^2} e(t,x)dx$ satisfies the energy balance law
\begin{equation}\label{eq:en_disp_apdx}
E(t)+\frac{1}{\tau}\int_0^t\int_{\T^2}\rho |v|^2dxds=E_0.
\end{equation}
\end{prop}
\begin{proof}
Since we are dealing with solutions with positive density $\rho>0$, we can write system \eqref{eq:QHD_md} as
\begin{equation}\label{eq:QHD_apdx}
\left\{\begin{aligned}
&\d_t\rho +\diver(\rho  v )=0,\\
&\d_tv + (v\cdot \nabla)v +\nabla f'(\rho )+\nabla V 
=\frac12\nabla\left(\frac{\triangle\sqrt{\rho }}{\sqrt{\rho }}\right)-\frac{1}{\tau}v.
\end{aligned}\right.
\end{equation}
We notice that, by the irrotationality condition $\operatorname{curl}v=0$, the equation for the velocity can be equivalently written as
\begin{equation*}
\d_tv +\nabla\mu +\frac{1}{\tau}v =0,
\end{equation*}
where $\mu$ is the chemical potential defined in \eqref{eq:chem}, and $\nabla\mu$ is interpreted in the $H^{-1}_x$ sense. 
By using the expression of $e(t,x)$, we can differentiate it with respect to time and find
\begin{equation*}
\begin{aligned}
\d_te =&\nabla\sqrt{\rho }\cdot\d_{t}\nabla\sqrt{\rho }+\left(\frac{1}{2}|v|^2 +f'(\rho )\right)\d_t\rho +\rho  v\cdot \d_tv +\nabla V\cdot \nabla\d_tV\\
=&\diver\left(\nabla\sqrt{\rho }\d_t\sqrt{\rho }\right)
+\left(-\frac12\frac{\triangle\sqrt{\rho }}{\sqrt{\rho }}+\frac{1}{2}|v| ^2+f'(\rho )+V \right)\d_t\rho \\
&+\rho  v \cdot\d_tv +V \triangle\d_tV +\nabla V\cdot \nabla\d_tV.
\end{aligned}
\end{equation*}
Using the QHD equations above and definition \eqref{eq:chem}, we then have
\begin{equation*}\begin{aligned}
\d_te =&\diver\left(\nabla\sqrt{\rho }\d_t\sqrt{\rho }\right)-\mu \diver(\rho  v )\\
&+\diver(V \nabla\d_tV )-\rho  v\cdot \nabla\mu -\frac{1}{\tau}\rho  |v| ^2\\
=&\diver\left(\nabla\sqrt{\rho }\d_t\sqrt{\rho }-\rho  v \mu +V \nabla\d_tV \right)-\frac{1}{\tau}\rho  |v|^2 .
\end{aligned}
\end{equation*}
Last, the energy balance law \eqref{eq:en_disp_apdx} is obtained by integrating \eqref{eq:en_cons_apdx} on the time interval $[0,t]$.
\end{proof}

We now turn to the time derivative of the functional $I(t)$ defined in \eqref{eq:higher}. In contrast with the results obtained in Section \ref{sect:exist2d}, which were rigorously established through the mollification of wave functions, the regularity of GCP solutions is not sufficient to justify the limiting procedure when one regularizes hydrodynamic variables such as $\rho$, $v$, and $\mu$. Whether there is an intrinsic distinction between the wave-function formulation and the purely hydrodynamic formulation is an interesting question. Here we restrict ourselves to a formal computation under the assumption of smooth solutions.

\begin{prop}\label{prop:dtI_apdx}
Let $(\rho , v )$ be a smooth solution to \eqref{eq:QHD_md} such that $\rho >0$. Then the time derivative of $I (t)$ is given by
\begin{align*}
\frac{d}{dt}I (t)+\frac{1}{\tau}\int_{\T^2} \rho \sigma ^2 dx=&\int_{\T^2}\mu \d_tp(\rho )dx+\int_{\T^2}\rho \mu \d_tV  dx-\frac{1}{\tau}\int_{\T^2}\rho  |v| ^2\mu \,dx.
\end{align*}
\end{prop}

\begin{proof}
By using the formula 
\begin{equation*}
\rho \mu =-\frac14\triangle\rho +e +p(\rho )-\frac12|\nabla V |^2+\rho  V 
\end{equation*}
together with \eqref{eq:en_cons_apdx}, we have
\begin{equation*}
\begin{aligned}
\d_t(\rho \mu )=&\d_t\left(e -\frac14\triangle\rho +p(\rho )-\frac12|\nabla V |^2+\rho  V \right)\\
=&\diver(\nabla\sqrt{\rho }\d_t\sqrt{\rho }-\rho  v \mu +V \nabla\d_tV )\\
&-\frac14\triangle\d_t\rho -\frac{1}{\tau}\rho  |v|^2 +\d_tp(\rho )-\nabla V \cdot \nabla\d_tV +V \d_t\rho +\rho \d_tV .
\end{aligned}
\end{equation*}
Again by using the continuity equation for $\rho$ and the Poisson equation for $V $, we can write
\begin{equation}\label{eq:mu_evol_apdx}
\begin{aligned}
\rho \d_t\mu +\rho  v \cdot\nabla\mu =\diver(\nabla\sqrt{\rho }\d_t\sqrt{\rho })-\frac14\triangle\d_t\rho \\
-\frac{1}{\tau}\rho  |v|^2 +\d_tp(\rho )+\rho \d_tV .
\end{aligned}
\end{equation}
Now, to write the equation for $\sigma $ we may proceed in the following way. By writing the continuity equation as below
\begin{equation*}
\d_t\rho +v \cdot\nabla\rho +\rho \diver v =0,
\end{equation*}
we find the equation for $\log\sqrt{\rho }$, namely
\begin{equation*}
\d_t\log\sqrt{\rho }+v\cdot\nabla\log\sqrt{\rho }+\frac12\diver v =0.
\end{equation*}
Differentiating the last equation with respect to time and using $\d_tv=-\nabla\mu-\frac{1}{\tau}v$, we obtain
\begin{equation*}
\d_t\sigma +v \cdot \nabla\sigma -\left(\nabla\mu +\frac{1}{\tau}v\right)\cdot\nabla\log\sqrt{\rho }-\frac12\diver\left(\nabla\mu +\frac{1}{\tau}v\right)=0.
\end{equation*}
By multiplying this by $\rho $ and using
\[
\rho  v \cdot\nabla\log\sqrt\rho +\frac12\rho \diver v =\frac12\diver(\rho  v )=-\rho \sigma ,
\]
we get
\begin{equation}\label{eq:sigma_evol_apdx}
\rho \d_t\sigma +\rho  v\cdot \nabla\sigma +\frac{1}{\tau}\rho \sigma =\frac12\diver(\rho \nabla\mu ).
\end{equation}
Now we can use the equations \eqref{eq:mu_evol_apdx} and \eqref{eq:sigma_evol_apdx} to compute the time derivative of the functional $I (t)$. After integrating by parts we obtain
\begin{align*}
\frac{d}{dt}I (t)=&\frac{1}{2}\int_{\T^2} (\mu ^2+\sigma ^2)\d_t\rho  dx-\frac{1}{2}\int_{\T^2} \rho  v\cdot \nabla(\mu ^2+\sigma ^2)dx\\
&-\frac{1}{2}\int_{\T^2} \mu \diver(\rho\nabla\sigma )dx+\frac{1}{2}\int_{\T^2}\sigma \diver(\rho \nabla\mu )dx\\
&+\int_{\T^2} \mu \d_tp(\rho )dx+\int_{\T^2} \rho \mu \d_tV  dx-\frac{1}{\tau}\int_{\T^2} \rho \sigma ^2dx-\frac{1}{\tau}\int_{\T^2}\rho  |v| ^2\mu  dx\\
=&\int_{\T^2} \mu \d_tp(\rho )dx+\int_{\T^2} \rho \mu \d_tV  dx-\frac{1}{\tau}\int_{\T^2} \rho \sigma ^2dx-\frac{1}{\tau}\int_{\T^2}\rho  |v| ^2\mu dx.
\end{align*}
\end{proof}

\section*{Acknowledgments}

This work was partially supported by the Strategic Priority Research Program of the Chinese Academy of Sciences under Grant No. XDB0510201 (H.Z.).


\begin{thebibliography}{100}

\bibitem{AI} M. Ancona, G. Iafrate, \emph{Quantum correction to the equation of state of an electron gas in a semiconductor}, Phys. Rev. B {\bf 39} (1989), 9536-9540.

\bibitem{AT} M.G. Ancona, H.F. Tiersten, \emph{Macroscopic physics of the silicon inversion layer}, Phys. Rev. B {\bf 35}, no. 15 (1987), 7959-7965.


\bibitem{ACLS} P. Antonelli, G. Cianfarani Carnevale, C. Lattanzio, S. Spirito., \emph{ Relaxation Limit from the Quantum Navier-Stokes Equations to the Quantum Drift-Diffusion Equation}, J. Nonlinear Sci., {\bf 31}, 71 (2021).


\bibitem{AM1} P. Antonelli, P. Marcati, \emph{On the finite energy weak solutions to a system in Quantum Fluid Dynamics}, Comm. Math. Phys. {\bf 287} (2009), no 2, 657--686.

\bibitem{AM2} P. Antonelli, P. Marcati, \emph{The quantum hydrodynamics system in two space dimensions}, Arch. Rat. Mech. Anal. \textbf{203} (2012), 499--527.

\bibitem{AM_b} P. Antonelli, P. Marcati, \emph{An introduction to the mathematical theory of quantum fluids}, to appear UMI Springer Lecture Notes.


\bibitem{AMZ1} P. Antonelli, P. Marcati, H. Zheng,  \textit{Genuine Hydrodynamic Analysis to the 1-D QHD system: Existence, Dispersion and Stability},  Comm. Math. Phys. {\bf 383}, (2021), 2113--2161.

\bibitem{AMZ2} P. Antonelli, P. Marcati, H. Zheng,  \textit{An Intrinsically Hydrodynamic Approach to Multidimensional QHD Systems},  Arch. Rat. Mech. Anal. {\bf 247}, no 24, (2023).

\bibitem{AMZ3} P. Antonelli, P. Marcati, H. Zheng,  \textit{The relaxation-time limit for weak solutions to the quantum hydrodynamics system},  Arch. Rat. Mech. Anal. {\bf 249}, no 73, (2025).




\bibitem{BG} S. Benzoni-Gavage, \emph{Propagating phase boundaries and capillary fluids}, available online at \url{http://math.univ-lyon1.fr/~benzoni/Levico.pdf}.

\bibitem{BW} G. Baccarani, M.R. Wordeman \emph{An investigation of steady-state velocity overshoot in silicon}, Solid-State Electronics, {\bf 28} (1985),407--416.


\bibitem{Ber} A. L. Bertozzi, \emph{The mathematics of moving contact lines in thin liquid films}, Notices Amer. Math. Soc., {\bf 45} (1998), 689--697.


\bibitem{BP1} A. L. Bertozzi, M. C. Pugh, \emph{Long-wave instabilities and saturation in thin film equations}, Comm. Pure Appl. Math., {\bf 51} (1998),  625--661.

\bibitem{DLSS1} B. Derrida, J. Lebowitz, E. Speer, and H. Spohn, \emph{Dynamics of an anchored Toom interface}, J. Phys. A, {\bf 24} (1991), 4805--4834.

\bibitem{DLSS2} B. Derrida, J. Lebowitz, E. Speer, and H. Spohn, \emph{Fluctuations of a stationary nonequilibrium interface}, Phys. Rev. Lett., {\bf 67} (1991), 165--168.

\bibitem{DM} D. Donatelli, P. Marcati, \emph{Convergence of Singular Limits for Multi-D Semilinear Hyperbolic Systems to Parabolic Systems}, Transactions of the American Mathematical Society, {\bf 365} No.5 (2004), 2093--2121.

\bibitem{DGJ} J. Dolbeault, I. Gentil, A. J\"ungel, \emph{A logarithmic fourth-order parabolic equation and related logarithmic Sobolev inequalities}, Comm. Math. Sci. {\bf 4} (2006),  No.2, 275--290


\bibitem{G} C. Gardner, \emph{The quantum hydrodynamic model for semiconductor devices}, SIAM J. Appl. Math. {\bf 54} (1994), 409-427.

\bibitem{Guo} B-L. Guo, \emph{Quantum Hydrodynamic Equation and Its Mathematical Theory}, World Scientific, (2023).

\bibitem{GST} U. Gianazza, G. Savar\'e, G. Toscani, \emph{The Wasserstein gradient flow of the fisher information and the quantum drift-diffusion equation}, Arch. Ration. Mech. Anal., {\bf 194} (2009), 133--220.

\bibitem{H} F. Haas, \emph{Quantum plasmas: An hydrodynamic approach}, New York: Springer (2011).

\bibitem{HC} E. Heifetz, E. Cohen, \emph{Toward a Thermo-hydrodynamic Like Description of Schr\"odinger Equation via the Madelung Formulation and Fisher Information}, Found. Phys. {\bf 45} (2015), 1514--1525.

\bibitem{RaHong} H. Hong, S. Ra \emph{The existence, uniqueness and exponential decay of global solutions in the full quantum hydrodynamic equations for semiconductors}, Z. Angew. Math. Phys. {\bf 72} (2021), Article ID 107, 32 pp.

\bibitem{HL} L. Hsiao, T.P. Liu, \emph{Convergence to nonlinear diffusion waves for solutions of a system of hyperbolic
conservation laws with damping}, Comm. Math. Phys. 143 (1992) 599-–605.

\bibitem{HP} F. Huang, R. Pan\emph{Asymptotic behavior of the solutions to the damped compressible Euler equations with vacuum}, J. Diff. Equa. {\bf 220} (2006), 207--233.

\bibitem{J} A. J\"ungel, \emph{Quasi-Hydrodynamic Semiconductor Equations}, Progress in Nonlinear Differential Equations, Birkh\"auser, Basel, 2001.

\bibitem{JP} A. J\"ungel and Y.-J. Peng, \emph{A hierarchy of hydrodynamic models for plasmas: zero relaxation-time limits}, Comm. Part. Diff. Eqs., 24 (1999), 1007–-1033.

\bibitem{JLi} A. J\"ungel, H.-L. Li, \emph{Quantum Euler-Poisson systems: Global existence and exponential decay}, Quart. Appl. Math. {\bf 62} (2004), 569--600. 

\bibitem{JLM} A. J\"{u}ngel, H-L. Li, A. Matsumura, \emph{The relaxation-time limit in the quantum hydrodynamic equations for semiconductors}, J. Diff. Equa. {\bf 225} (2) (2006), 440--464.

\bibitem{JMR} A. J\"ungel, M.C. Mariani, D. Rial, \emph{Local existence of solutions to the transient quantum hydrodynamic equations}, Math. Mod. Meth. Appl. Sci. {\bf 12} (2002), 485.

\bibitem{LiMarcati} H. Li, P. Marcati, \emph{Existence and asymptotic behavior of multi-dimensional quantum hydrodynamic model for semiconductors}, Comm. Math. Phys. {\bf 245} (2004), 215--247.

\bibitem{JM} A. J\"{u}ngel, D. Matthes, \emph{The Derrida-Lebowitz-Speer-Spohn Equation: Existence, NonUniqueness, and Decay Rates of the Solutions}, SIAM Journal on Mathematical Analysis, {\bf 39}(6), (2008), 1996--2015.

\bibitem{K} I. Khalatnikov, \emph{An Introduction to the theory of Superfluidity}, (2000).

\bibitem{KK} B. Kwak, S. Kwon, \emph{Critical local well-posedness of the nonlinear Schr\"odinger equation on the torus}, Annales de l'Institut Henri Poincar\'e C, (2024)

\bibitem{Kos} M. D. Kostin, \emph{On the Schr\"odinger--Langevin equation}, J. Chem. Phys. 57:3589 (1972); J. Stat. Phys. 12:145 (1975).

\bibitem{LL} L.D. Landau, E.M. Lifshitz, \emph{Course of Theoretical Physics, vol. 6. Fluid Mechanics}, 2nd edition, Elsevier 1987.

\bibitem{LT} C. Lattanzio, A. Tzavaras, \emph{From gas dynamics with large friction to gradient flows describing diffusion theories}, Commun. Part. Diff. Equ. {\bf 42}, no. 2 (2017), 261--290.

\bibitem{LZZ} H.-L. Li, G. Zhang, K. Zhang, \emph{Semiclassical and relaxation limits of bipolar quantum hydrodynamic model for semiconductors}, J. Diff. Equ. {\bf 245} (2008), 1433-–1453.

\bibitem{Mad} E. Madelung, \textit{Quantentheorie in hydrodynamischer form}, Z. Physik {\bf 40} (1927), 322.

\bibitem{MN} P. Marcati, R. Natalini, \emph{Weak solutions to a hydrodynamic model for semiconductors and relaxation to the drift-diffusion equation}, Arch. Rational Mech. Anal. {\bf 129}, 129--145 (1995).

\bibitem{MMS} P. Marcati, A.J. Milani, P. Secchi, \emph{Singular convergence of weak solutions for a quasilinear nonhomogeneous hyperbolic system}, Manuscripta Math {\bf 60}, 49--69 (1988).

\bibitem{Na} A. B. Nassar, \emph{Fluid formulation of a generalized Schrodinger-Langevin equation} J. Phys. A: Math. Gen. {\bf 18} (1985), L509.

\bibitem{PS} L. Pitaevskii, S. Stringari, \emph{Bose-Einstein condensation and superfluidity}, Clarendon Press, Oxford, (2016).


\bibitem{R} W.V. Roosbroeck, \emph{Theory of flow of electrons and holes in germanium and other semiconductors}, Bell. Syst. Techn. J., {\bf 29} (1950), 560--607.

\bibitem{Y} K. Yasue, \emph{A note on the derivation of the Schr\"odinger--Langevin equation}. J. Stat. Phys. {\bf 16}, 113-116 (1977).

\bibitem {HaoMinE} H. Zheng, \emph{The Pauli problem and wave function lifting: reconstruction of quantum states from physical observables}, Math. in Eng. {\bf 6}, no. 4 (2024), 648--675.

\end{thebibliography}
\end{document}